\documentclass[final,notitlepage,8pt]{article}
\usepackage{indentfirst}
\usepackage[normalem]{ulem}
\usepackage{float}

\usepackage{multirow}
\usepackage{graphicx}
\usepackage{enumitem}[2011/09/28]
\setenumerate{align=left, leftmargin=0pt,labelsep=.5em, labelindent=0\parindent,listparindent=\parindent,itemindent=*}
\usepackage{array}
\usepackage{empheq}
\usepackage{natbib}
\usepackage[all]{xy}

\usepackage{exscale,relsize}
\usepackage{multirow}

\usepackage{caption}[2012/02/19]

\usepackage{longtable}

\usepackage{fouriernc}
\usepackage{fourier}
\usepackage{amssymb,bm,amsmath}
\usepackage{amsfonts}

\usepackage[OT2,T1,T2A]{fontenc}

\usepackage[english,french,german,italian,russian]{babel}
\usepackage{appendix}

\DeclareMathOperator{\D}{d}
\DeclareMathOperator{\Tr}{Tr}
\DeclareMathOperator{\md}{d}
\DeclareMathOperator{\supp}{supp}

\DeclareMathOperator{\R}{Re}

\bibpunct{[}{]}{,}{n}{}{;}

\def\qed{\hfill$ \blacksquare$}
\def\eor{\hfill$ \square$}

\newtheorem{theorem}{Theorem}[section]

\newtheorem{proposition}[theorem]{Proposition}

\newtheorem{example}[theorem]{Example}

\newenvironment{proof}[1][Proof]{\begin{trivlist}
\item[\hskip \labelsep {\bfseries #1}]}{\end{trivlist}}

\newenvironment{remark}[1][Remark]{\begin{trivlist}
\item[\hskip \labelsep {\bfseries #1}]}{\end{trivlist}}

\begin{document}

\selectlanguage{english}
\title{\textsf{{ A Simple Formula for Scalar Curvature of Level Sets in Euclidean Spaces}}}\author{Yajun Zhou
\\\begin{small}\textsc{Program in Applied and Computational Mathematics, Princeton University, Princeton, NJ 08544}\end{small}
}
\date{}
\maketitle

\def\abstractname{}
\begin{abstract}
     \noindent A simple formula is derived for the Ricci scalar curvature of any smooth level set $\{ \psi(x_0,x_1,\dots,x_n)=C\}$ embedded in the Euclidean space $ \mathbb R^{n+1}$, in terms of the gradient $ \nabla\psi$ and the Laplacian $ \Delta\psi$.  Some applications are given to the geometry of low-dimensional $p$-harmonic functions and high-dimensional harmonic functions. \end{abstract}

\setcounter{section}{-1}
\section{Introduction}

Many problems in mathematical physics can be described by  ``inhomogeneous Laplace equations'' in the Euclidean space $ \Delta\psi(\bm r)=V(\bm r)$, which not only encompass the familiar examples of linear equations (\textit{e.g.}~the Poisson equation in electrostatics, the Helmholtz equation in wave propagation, the Schr\"odinger equation in non-relativistic quantum mechanics) but also some important classes of nonlinear partial differential equations (including, but not limited to the Landau-Ginzburg model for superconductivity, the Gross-Pitaevski equation for Bose-Einstein condensates and the Navier-Stokes equations for viscous fluids).

In this series of works, we study the geometry of level sets associated with solutions to the  ``inhomogeneous Laplace equations''  $ \Delta\psi(\bm r)=V(\bm r)$, starting with the current article which revolves around a simple formula for the Ricci scalar curvature $ \mathcal R(\bm r)$ at any point on a smooth level set $ \{\psi(\bm r)=C\}$ embedded in the Euclidean space of arbitrary dimension:
\begin{align*}
\mathcal R(\bm r)=-\Delta\log |\nabla\psi(\bm r)|+\nabla\cdot\left[ \Delta\psi(\bm r)\frac{\mathcal \nabla\psi(\bm r)}{ |\nabla\psi(\bm r)|^{2}} \right].
\end{align*}
The right-hand side of the formula above invokes the gradient $ \nabla\psi(\bm r)$ and the Laplacians $ \Delta\psi(\bm r)$, $ \Delta\log|\nabla\psi(\bm r)|$ evaluated in the Cartesian coordinate system of the ambient Euclidean space. Likewise, the divergence operator $ \nabla\cdot$ also refers to the flat space formulation. By convention, we will restrict our attention to cases where the gradient is non-vanishing $ \nabla\psi(\bm r)\neq\mathbf0$, so that the appearance of $ |\nabla\psi(\bm r)|^2$ in the denominator is meaningful.

We call the proposed scalar curvature identity a ``simple formula'' because it only involves some familiar operations and quantities of physical interest.  Especially, in many practical problems, the modulus of the gradient $ |\nabla\psi(\bm r)|$ provides a measure of the ``field intensity'', and $ \Delta\psi(\bm r)=V(\bm r)$ is given as a prescribed function in the Euclidean space or as a certain function(al) of $ \psi(\bm r)$. The simple formula in the preceding paragraph thus allows us to deduce some useful information regarding the curvature properties of the level sets and their relations to ``field intensities'', which we will describe in several papers in this series.

In the current work, we open with  generic formulations for the geometry of level sets (\S\S\ref{subsec:metric}-\ref{subsec:Lap_evolv}), and follow the proof of the aforementioned scalar curvature identity (\S\ref{subsec:curv_id_proof}) with a brief discussion on its invariance properties (\S\ref{subsec:diff_inv}). We will then show that the degenerate cases of the scalar curvature identity in low dimensions recover some familiar results in classical analysis (\S\ref{subsec:CA}) and are consistent with various worked examples in classical physics (\S\ref{subsec:CP}). We close this proof-of-principle article with a simple low-dimensional application to $ p$-harmonic functions (\S\ref{sec:p_harmonic}) and a high-dimensional application to harmonic functions (\S\ref{sec:harmonic_sect_curv}).
\section{A Simple Scalar Curvature Formula for Level Sets\label{sec:curv_proof}}
\subsection{Metric, Connection and Curvatures for Level Sets Embedded in Euclidean Spaces\label{subsec:metric}  }A natural way to describe the local geometric properties of a smooth level set $ \{\psi(x_0,x_1,\dots,x_n)=C\}$  as a hypersurface in $ \mathbb R^{n+1}$ is to introduce the \textit{line element for curves on the hypersurface} as $\md
s^2=g_{ij}\md u^i\md u^j$, where a Latin index takes values from $1$ to $n$, and
every repeated index implies summation \cite[Ref.][\S7.3, pp.~68-69]{Dubrovin1}. Here,
$(u^1,\dots,u^n)$ are local curvilinear coordinates on the surface (as opposed to the Cartesian coordinates $ (x_0,x_1,\dots ,x_n)\in\mathbb R^{n+1}$), and $(g_{ij})$
is the (covariant) \textit{metric tensor}.
Superscript and subscript indices stand for contravariant and covariant components,
respectively. Written explicitly, the components of the metric tensor takes
the form $g_{ij}=\partial_i\bm r\cdot\partial_j\bm r$, where $\partial_i$
is a short-hand for $\partial/\partial u^i$. The contravariant metric tensor
$(g^{ij})$ is simply defined as the matrix inverse of $(g_{ij})$.

The tangent vectors $\partial_i\bm r $ at different points on the level set in question are connected by the {Gauss} formula \cite[Ref.][\S30.4, p.~311]{Dubrovin1}: $\partial_i\partial_j\bm
 r=\Gamma_{ij}^k\partial_k\bm r+b_{ij}\bm n $, where $\bm n$ is the unit normal
 vector of the hypersurface, the \textit{connection coefficients} are $\Gamma_{ij}^k=g^{k\ell}\partial_i\partial_j\bm r\cdot\partial_\ell\bm r$ and the \textit{coefficients of second fundamental form} are $b_{ij}=\partial_i\partial_j\bm r\cdot\bm n$. The connection coefficients can also be evaluated from the
Christoffel formula \cite[Ref.][\S29.3, p.~293]{Dubrovin1}: $\Gamma_{ij}^k=\frac12g^{k\ell}(\partial_ig_{\ell
j}+\partial_jg_{i\ell}-\partial_\ell g_{ij})$.

The components of the Weingarten transform  $\hat W=(b_i^j)$ is defined by $b_i^j=g^{jk}b_{ki}$
and appears in the Weingarten formula: $\partial_i\bm n=-b_i^j\partial_j\bm
r $, that is, $\md \bm n=-\hat W \md\bm r$ for infinitesimal changes tangent to the surface.
The \textit{mean curvature} is equal to  $ \frac1n$ times the trace of the Weingarten transform: $H:=\frac1n\Tr(\hat W)=\frac1ng^{ij}b_{ij}$.

The expression $R_{ij}=\partial_\ell\Gamma^\ell_{ij}-\partial_j\Gamma^\ell_{i\ell}+\Gamma^\ell_{ij}\Gamma^m_{\ell
m}-\Gamma^m_{i\ell}\Gamma^\ell_{jm}$ gives the $(i,j)^\mathrm{th}$ component of \textit{Ricci tensor} on the  hypersurface \cite[Ref.][\S37.4, p.~394]{Dubrovin1}, and the \textit{Ricci scalar curvature} is defined by $ \mathcal R=g^{ij}R_{ij}$.

For convenience, we shall borrow some notations and  terminologies from physics: ``force field'' is defined via the negative gradient $ \bm F(\bm r):=-\nabla\psi(\bm r)$; ``field intensity'' $ F:=|\bm F|$ is the modulus of the gradient; ``field line'' or ``$ \bm F$-line'' is the integral curve of the vector field $ \bm F(\bm r),\bm r\in\mathbb R^{n+1}$. By convention, we orient the normal vector
as $ \bm n=\bm F/ F$.

 To extend the discussion on an isolated level set ($n$-dimensional manifold) to a family
of hypersurfaces that fill a region in the ambient $ (n+1)$-dimensional Euclidean space,
we can employ the natural curvilinear coordinate system $\bm r(u^0,u^1,\dots,u^n)$
as follows: \begin{enumerate}
\item  For every point $\bm r$, the zeroth component of curvilinear coordinate
coincides with the ``level value'': $u^0=\psi(\bm r)$, thus the coordinate
$u^0$ is constant  on any level set; \item  Two distinct points
$\bm r(\psi,u^{1},\dots,u^n)$ and $\bm r(\psi',u'^1,\dots ,u'^n)$ can be linked by the integral curve ($\bm
F$-line) if and only if $u^{j}=u'^j,j=1,\dots,n$, so the $\bm F$-line acts as the curvilinear
coordinate
curve of $u^0$. Accordingly, $\bm F=-\nabla\psi$ is equivalent to $\partial_0\bm
r=-\bm n/F=-\bm F/F^{2} $.

\end{enumerate}

From the above curvilinear coordinate system constructed in a neighborhood surrounding a non-critical point $\bm r$ where $F(\bm r)=|\nabla\psi(\bm r)|\neq0 $, one can  define \textit{line element for curves in $(n+1)$-dimensional Euclidean space }as $\md
\underline s^2=g_{\mu\nu}\md u^\mu\md u^\nu$ where a Greek index takes values
$0,1,\dots,n$. By definition, $g_{00}=F^{-2}=1/g^{00}$ and $g_{0i}=g^{0i}=0$, so the decomposition of the Euclidean metric is a direct sum of the ``level value coordinate'' and ``level set metric''. One may extend
the definition of connection coefficients as $\partial_\mu\partial_\nu\bm
 r=\Gamma_{\mu\nu}^\lambda\partial_\lambda\bm r$, where the newly-arisen connection coefficients will be computed in the following proposition.
\begin{proposition}[Connection Coefficients]\label{prop:conn0} Suppose that a real-valued three-times continuously differentiable function $ \psi\in C^3(\mathfrak D;\mathbb R)$  satisfies the ``inhomogeneous Laplace equation'' \[\Delta\psi(\bm r)=V(\bm r),\quad \bm r\in \mathfrak D\]in a certain domain $\mathfrak D$ of the Euclidean space $ \mathbb R^{n+1}$, and has non-vanishing gradients therein $ F(\bm r):=|\nabla\psi(\bm r)|\neq0,\bm r\in\mathfrak  D\subset\mathbb R^{n+1}$. We endow the level sets of $\psi(\bm r)$  with unit normal $ \bm n=-\nabla \psi/|\nabla\psi|$ and decompose the Euclidean metric $ \md
\underline s^2=g_{\mu\nu}\md u^\mu\md u^\nu=F^{-2}(\D\psi)^2+g_{ij}\D u^i\D u^j$ into a direct sum of the ``level value coordinate'' and ``level set metric'', then we have the following computations for connection coefficients involving the index $0$:\begin{align}&\Gamma _{ij}^{0} =-Fb_{ij},\quad\Gamma _{j0}^{0} ={}-\frac{1}{F}\frac{\partial F}{\partial u^{j}}, \quad
\Gamma _{j0}^{k} =\frac{b_{j}^{k}}{F};\label{eq:Gamma1}\\&\Gamma _{ij}^{0} =-\frac{F^2}{2}\frac{\partial g_{ij}}{\partial\psi\ },\quad\Gamma _{00}^{j}=-\frac{1}{2}g^{jm}\frac{\partial g_{00}}{\partial u^{m}}=%
\frac{1}{F^{3}}g^{jm}\frac{\partial F}{\partial u^{m}};\label{eq:Gamma2}\\&\Gamma^0_{00}=\frac{1}{2}g^{00}\partial_0g_{00}=-\frac{\partial}{\partial\psi}\log F,\label{eq:Gamma3}\end{align}along with a modified version of the ``harmonic coordinate condition'':\begin{align}\label{eq:mod_harmonic_cond}
\Gamma^{0}:=g^{\mu\nu}\Gamma_{\mu\nu}^0=F^{2}(\Gamma^{0}_{00}-\Gamma^{m}_{m0})=-V.
\end{align}\end{proposition} \begin{proof}
 To prove the three identities in Eq.~\ref{eq:Gamma1}, it would suffice to compare the equation $\partial_\mu\partial_\nu\bm
 r=\Gamma_{\mu\nu}^\lambda\partial_\lambda\bm r$ with the
 Gau{ss}  and Weingarten formulae: \begin{align*}\frac{\partial ^{2}{\bm r}}{\partial u^{i}\partial u^{j}} =\Gamma _{ij}^{k}%
\frac{\partial {\bm r}}{\partial u^{k}}-Fb_{ij}\frac{\partial {\bm r}}{%
\partial \varphi }, \qquad
\frac{\partial }{\partial u^{j}}\left( F\frac{\partial {\bm r}}{\partial
\varphi }\right) =b_{j}^{k}\frac{\partial {\bm r}}{\partial u^{k}}.\end{align*}
The two identities in Eq.~\ref{eq:Gamma2}, as well as  Eq.~\ref{eq:Gamma3},  follow from the
Christoffel formula  $\Gamma_{\mu\nu}^\lambda=\frac12g^{\lambda\eta}(\partial_\mu
g_{\eta
\nu}+\partial_\nu g_{\mu\eta}-\partial_\eta g_{\mu\nu})$.

Juxtaposing the two expressions of  $\Gamma_{ij}^0$ in Eqs.~\ref{eq:Gamma1} and \ref{eq:Gamma2}, we can put down\begin{align}\label{eq:4HE}\frac{\partial g_{ij}}{\partial\psi
}=\frac2Fb_{ij},\quad\frac{\partial\log\det( g_{ij})}{\partial \psi
}=g^{ij}\frac{\partial g_{ij}}{\partial \psi
}=\frac2Fg^{ij}b_{ij}=\frac{2nH}{F}=\frac{2\Tr(\hat W)}{F}.\end{align}
On the other hand, the
``inhomogeneous Laplace equation''  $ \Delta\psi(\bm r)=-\nabla\cdot\bm F(\bm r)=V(\bm r) $ prescribes the divergence  of
force field $ \bm F(\bm r)$, \textit{i.e.} $(F/\sqrt{g})\partial_0(F\sqrt{g})
=V$ (hereafter $g=\det(g_{ij}) $), which can be combined with  Eqs.~\ref{eq:Gamma1} and \ref{eq:4HE} into \begin{align*}\Gamma^{0}:=g^{00}\Gamma^0_{00}+g^{ij}\Gamma_{ij}^0=F^{2}(\Gamma^{0}_{00}-\Gamma^{m}_{m0})=-\frac{F}{\sqrt{g}}\frac{\partial}{\partial\psi}(F\sqrt g)=-\nabla\cdot\nabla\psi=-V,\end{align*}as stated in  Eq.~\ref{eq:mod_harmonic_cond}. If $ V(\bm r)\equiv0$, and $\psi(\bm r)$ is a harmonic function, then the formula $ \Gamma^0=g^{\mu\nu}\Gamma_{\mu\nu}^0=0$ hearkens back to the ``harmonic coordinate condition'' in general relativity.
   \qed
\end{proof}\begin{remark}From the identity $\partial_0g_{ij}=2b_{ij}/F $, we can  also readily  deduce $\partial_0 g^{ij}=-2g^{ik}b_k^j/F$.

Another by-product of the foregoing argument  is the following result that will be  used later in this work:\begin{align}\Tr(\hat W)+\frac{\partial F}{\partial\psi}=\frac{V}{F},\quad\text{i.e.~}(\bm n\cdot\nabla)\log F=\Tr(\hat W)-\frac{V}{F}.\label{eq:nH}\end{align}When $V=0$ and $n=2$, the formula above is a standard exercise in electrostatics \cite[Ref.][Exercise~1.11]{Jackson:EM}.\eor\end{remark}

\subsection{Laplace Operator in Curvilinear Coordinates and Evolution of Mean Curvature\label{subsec:Lap_evolv}}The Laplace operator on the Euclidean space $\mathbb R^{n+1} $ can be presented in curvilinear coordinates
as \[\Delta=g^{\mu\nu}(\partial_\mu\partial_\nu-\Gamma_{\mu\nu}^\lambda\partial_\lambda)=\frac1{\sqrt{\smash[b]{\det(g_{\mu\nu})}}}\partial_\lambda\left( g^{\lambda\eta}\sqrt{\smash[b]{\det(g_{\mu\nu})}}\partial_\eta \right).\]
Here, $\det(g_{\mu\nu})=g/F^{2}$. Similarly, one can define the \textit{Laplace operator on the level set $\Sigma$}
as \[\Delta _{\Sigma }=g^{ij}( \partial_i\partial_j-\Gamma_{ij}^k\partial_k)=\frac1{\sqrt{g}}\partial_k\left( g^{k\ell}\sqrt{g}\partial_\ell \right).\]

\begin{proposition}[Decomposition of Laplacian]\label{prop:Lap_decomp}The Laplace operator $ \Delta$ can be rewritten as\begin{align}\Delta =\Delta _{\Sigma }+F^{2}\frac{\partial ^{2}}{\partial\psi ^{2}}-\frac{1%
}{F}g^{jm}\frac{\partial F}{\partial u^{m}}\frac{\partial }{\partial u^{j}}+V\frac{\partial }{\partial
\psi },\label{eq:F_LapOp}\end{align}where\[\Delta_\Sigma:=\frac{1}{\sqrt{g}}\partial_i(g^{ij}\sqrt{g}\partial_j)=g^{ij}(\partial_i\partial_j-\Gamma^k_{ij}\partial_k)\]is the Laplace-Beltrami operator on the level set of $ \psi$. Accordingly, we have the following formula\begin{align}
\Delta \log F=-F\Delta_\Sigma\frac1F+F^2\frac{\partial ^{2}\log F}{\partial\psi ^{2}}+\Gamma^0\Gamma_{00}^0=-F\Delta_\Sigma\frac1F+F^2\frac{\partial ^{2}\log F}{\partial\psi ^{2}}-\frac{
\Tr(\hat W)V}{F}+\frac{V^{2}}{F^{2}}. \label{eq:Lap_lnF_mod}
\end{align}\end{proposition}\begin{proof}By definition, we have\[\Delta =\Delta _{\Sigma }-g^{ij}\Gamma _{ij}^{0}\frac{\partial }{\partial
\psi }+g^{00}\left( \frac{\partial ^{2}}{\partial \psi^{2}}-\Gamma
_{00}^{j}\frac{\partial }{\partial u^{j}}-\Gamma _{00}^{0}\frac{\partial }{%
\partial \psi }\right)=\Delta _{\Sigma }+g^{00}\left( \frac{\partial ^{2}}{\partial \psi^{2}}-\Gamma
_{00}^{j}\frac{\partial }{\partial u^{j}}\right)-\Gamma^{0}\frac{\partial }{\partial
\psi }.\]With the substitution of $g^{00}=F^2$ and the expressions for  $\Gamma _{ij}^0,\Gamma^j_{00},\Gamma^0_{00},\Gamma^0 $ from Proposition~\ref{prop:conn0}, we obtain the claimed result in Eq.~\ref{eq:F_LapOp}. To derive Eq.~\ref{eq:Lap_lnF_mod}, it would suffice to spell out the term $ \Gamma^0\Gamma_{00}^0$ using Eq.~\ref{eq:nH}.\qed\end{proof}\begin{remark}By the relations $ \Delta\bm r=\mathbf 0$ and $ \Delta_\Sigma\bm r=\Tr(\hat W)\bm n$, one may also use Eq.~\ref{eq:F_LapOp} to deduce an explicit expression for the second order derivative $\partial^2\bm r/\partial\psi^2 $. \eor\end{remark}

For a hypersurface embedded in Euclidean space, its Ricci tensor $ (R_{ij})$ and Ricci scalar curvature $ \mathcal R$ can be spelt out ``extrinsically'' in terms of  the Weingarten transformation and the principal curvatures. Concretely speaking,
in the principal curvature coordinate system where the Weingarten transformation $ \hat W=(b^{i}_j)$ is represented by a diagonal matrix with eigenvalues $ k_1,\dots,k_n$ (\textit{viz.}~the $n$ principal curvatures of the hypersurface), one has \begin{align*}(R_{ij})_{1\leq i,j\leq n}=\begin{pmatrix}g_{11}[(k_1+\cdots+k_n)k_1-k_1^2] &  &
\multicolumn{2}{c}{\raisebox{-2.3ex}[0pt]{\Huge0}}
\\
 & \ddots & &\\
\multicolumn{2}{c}{\raisebox{.8ex}[0pt]{\Huge0}}
 &  & g_{nn}[(k_1+\cdots+k_n)k_n-k_n^2]
\end{pmatrix}=(b^{k}_kb_{ij}-b_{ki}g^{km}b_{mj})_{1\leq i,j\leq n}.\end{align*}Here, to identify the two extreme ends of the equation above, we may recall that the geometric relation \[b_{ki}g^{km}b_{mj}-b^{k}_kb_{ij}=-R_{ij}\]applies to  any curvilinear coordinate system on any  $ n$-dimensional hypersurface embedded in $ \mathbb R^{n+1}$, as evident from the  contraction $ R_{ij}=R^k_{ikj}$ of the Riemann curvature tensor $ R^k_{i\ell j}=b_{ij}b^{k}_{\ell}-b_{i\ell}b^{k}_{j}$. Consequently, the Ricci scalar curvature  $ \mathcal R=g^{ij}R_{ij}=b^{i}_ib^j_i-b^i_jb^j_i$ can be explicitly given in terms of moments of principal curvatures $ \mathcal R=(\sum_{\ell=1}^n k_\ell)^2-\sum_{m=1}^n k^2_m=[\Tr(\hat W)]^{2}-\Tr(\hat W^2)$.

In the proposition below, both the surface Laplacian $ \Delta_\Sigma$ and the  Ricci scalar curvature $ \mathcal R$ turn up in the evolution of the mean curvature  along the ``level value coordinate''. \begin{proposition}[Evolution of the Second Fundamental Form]\label{prop:bij_H}We have the following identities
\begin{align}\frac{\partial b_{ij}}{\partial \psi }=( b_{j}^{k}b_{ki}- \partial_i\partial_j+\Gamma_{ij}^k\partial_k) \frac{1}{F}\label{eq:bij_evolv}\end{align}and\begin{align}\frac{\partial
\Tr(\hat W)}{\partial\psi}=-\Delta_\Sigma\frac1F-\frac{
\Tr(\hat W^{2})}{F}=-\Delta_\Sigma\frac1F-\frac{
[\Tr(\hat W)]^{2}-\mathcal R}{F},\label{eq:H_evolv}\end{align}where $\mathcal R$ is the Ricci scalar curvature of the level set in question.\end{proposition}\begin{proof}From the identity $
\partial_0(\partial_i\partial_j\bm
r)=\partial_i(\partial_0\partial_j\bm r)$, we may deduce   $\partial_0 \Gamma _{ij}^{0}+\Gamma _{ij}^{\nu }\Gamma _{\nu
0}^{0}={\partial_i \Gamma _{j0}^{0}}+\Gamma _{j0}^{\nu }\Gamma
_{\nu i}^0$. This results in Eq.~\ref{eq:bij_evolv}, upon substitution of the connection coefficients. Combining $\Tr(\hat W)=g^{ij}b_{ij}$ and  $\partial_0 g^{ij}=-2g^{ik}b_k^j/F$
with Eq.~\ref{eq:bij_evolv}, we obtain\begin{equation*}\frac{\partial
\Tr(\hat W)}{\partial\psi}=-\Delta_\Sigma\frac1F-\frac{b_j^kb_k^j}{F},\quad\text{where  } b_j^kb_k^j=\Tr(\hat W^2)=[\Tr(\hat W)]^{2}-\mathcal R . \end{equation*}This verifies Eq.~\ref{eq:H_evolv}.\qed\end{proof}\begin{remark}We originally discovered the scalar curvature identity mentioned in the introduction by an examination of the  evolution equation for mean curvature (Eq.~\ref{eq:H_evolv}). We leave it to the interested readers to recover such an ``extrinsic proof''. In the next subsection, we will only describe an ``intrinsic approach'' based on the definition of Ricci scalar curvature via the metric and connection on the level set. As will be explained elsewhere, the  intrinsic method has better extendibility when the ambient space is not Euclidean, and/or the co-dimension of the submanifold is higher than one.   \eor\end{remark}
\subsection{Scalar Curvature Identity and Evolution of Ricci Curvature\label{subsec:curv_id_proof}}We are now ready to prove the formula \begin{align*}
\mathcal R(\bm r)=-\Delta\log |\nabla\psi(\bm r)|+\nabla\cdot\left[ \Delta\psi(\bm r)\frac{\mathcal \nabla\psi(\bm r)}{ |\nabla\psi(\bm r)|^{2}} \right]
\end{align*}mentioned in the introduction. \begin{proposition}[Scalar Curvature Identity]Let $ F=|\nabla\psi|$ and $ V=\Delta\psi $, then we have the following geometric identity for the scalar curvature of the level sets for $ \psi$: \begin{align}\label{eq:dlnF_R} \Delta\log F+\mathcal R-\frac{\partial V}{\partial\psi}-\frac{
2\Tr(\hat W)V}{F}+\frac{V^{2}}{F^{2}}=\Delta\log F+\mathcal R-F^{2}\frac{\partial }{\partial\psi}\frac{V}{F^{2}}-\frac{V^{2}}{F^{2}}=\Delta\log F+\mathcal R-\nabla\cdot\left( \frac{V\nabla\psi}{F^{2}} \right)=0.\end{align}\end{proposition}\begin{proof}We define the Ricci tensor for the local curvilinear coordinates $(u^0=\psi,u^1,\dots,u^n) $ by \[\underline R_{\mu\nu}=\partial_\lambda\Gamma^\lambda_{\mu\nu}-\partial_\nu\Gamma^\lambda_{\mu\lambda}+\Gamma^\lambda_{\mu\nu}\Gamma^\tau_{\lambda
\tau}-\Gamma^\tau_{\mu\lambda}\Gamma^\lambda_{\nu\tau},\] so that  $ g^{\mu\nu}\underline{R}_{\mu\nu}=0$ follows from the flatness of the Euclidean space $ \mathbb R^{n+1}$. Meanwhile, by the direct sum decomposition of the metric $ g_{\mu\nu}$, the following identity holds\begin{align}0=g^{\mu\nu}\underline{R}_{\mu\nu}=F^2( \partial_\lambda\Gamma^\lambda_{00}-\partial_0\Gamma^\lambda_{0\lambda}+\Gamma^\lambda_{00}\Gamma^\tau_{\lambda
\tau}-\Gamma^\tau_{0\lambda}\Gamma^\lambda_{0\tau} )+\mathcal R+g^{ij}( \partial_0\Gamma^0_{ij}-\partial_j\Gamma^0_{i0}+\Gamma^\lambda_{ij}\Gamma^0_{\lambda
0}+\Gamma^0_{ij}\Gamma^m_{0m}-\Gamma^0_{i\lambda}\Gamma^\lambda_{j0} -\Gamma^m_{i0}\Gamma^0_{jm}).\label{eq:F_flat}\end{align}Using the modified version of the harmonic coordinate condition $ \Gamma^0:=g^{\mu\nu}\Gamma^0_{\mu\nu}=-V$ (Eq.~\ref{eq:mod_harmonic_cond}), we may derive\[F^2\Gamma^\lambda_{00}\Gamma^\tau_{\lambda
\tau}+g^{ij}( \Gamma^\lambda_{ij}\Gamma^0_{\lambda
0}+\Gamma^0_{ij}\Gamma^m_{0m} )=F^2\Gamma^m_{00}\Gamma^\tau_{m
\tau}+g^{ij}\Gamma^m_{ij}\Gamma^0_{m
0}+\Gamma^0(\Gamma^0_{00}+\Gamma^m_{0m} )=F^2\Gamma^m_{00}\Gamma^\tau_{m
\tau}+g^{ij}\Gamma^m_{ij}\Gamma^0_{m
0}-V(\Gamma^0_{00}+\Gamma^m_{0m} ).\]
Combining this result with the identity    $\partial_0 \Gamma _{ij}^{0}+\Gamma _{ij}^{\nu }\Gamma _{\nu
0}^{0}={\partial_i \Gamma _{j0}^{0}}+\Gamma _{j0}^{\nu }\Gamma
_{\nu i}^0$, we may rewrite Eq.~\ref{eq:F_flat} as{\allowdisplaybreaks[3]\begin{align*}0={}&F^2( \partial_m\Gamma^m_{00}-\partial_0\Gamma^m_{0m}+\Gamma^m_{00}\Gamma^\tau_{m
\tau}-\Gamma^\tau_{0\lambda}\Gamma^\lambda_{0\tau} )+\mathcal R+g^{ij}( \Gamma _{i0}^{\nu }\Gamma
_{j\nu }^0-\Gamma _{ij}^{\nu }\Gamma _{\nu
0}^{0}+\Gamma^m_{ij}\Gamma^0_{m
0}-\Gamma^0_{i\lambda}\Gamma^\lambda_{j0} -\Gamma^m_{i0}\Gamma^0_{jm})-V(\Gamma^0_{00}+\Gamma^m_{0m} )\notag\\={}&F^{2}( \partial_m\Gamma^m_{00}-\partial_0\Gamma^m_{0m}+\Gamma^m_{00}\Gamma^\tau_{m
\tau}-\Gamma^\tau_{0\lambda}\Gamma^\lambda_{0\tau} )+\mathcal R-g^{ij}( \Gamma _{ij}^{0 }\Gamma _{0
0}^{0}+\Gamma^0_{im}\Gamma^m_{j0})-V(\Gamma^0_{00}+\Gamma^m_{0m} )\notag\\={}&F^2( \partial_m\Gamma^m_{00}-\partial_0\Gamma^m_{0m}+\Gamma^m_{00}\Gamma^\tau_{m
\tau}-\Gamma^m_{0\lambda}\Gamma^\lambda_{0m} -\Gamma^0_{0m}\Gamma^m_{00})+\mathcal R-g^{ij}\Gamma^0_{im}\Gamma^m_{j0}-V\Gamma^m_{0m} \notag\\={}&F^{2}\left[\frac{\partial^{2}}{\partial\psi^{2}}\log F-\partial_0(\Gamma^m_{0m}-\Gamma^{0}_{00})\right]+F^2( \partial_m\Gamma^m_{00}+\Gamma^m_{00}\Gamma^\ell_{m
\ell} -\Gamma^0_{0m}\Gamma^m_{00})+\mathcal R-V\Gamma^m_{0m}\notag\\={}&F^{2}\left[\frac{\partial^{2}}{\partial\psi^{2}}\log F-\frac{\partial}{\partial\psi}\frac{V}{F^{2}}\right]+F^2( \partial_m\Gamma^m_{00}+\Gamma^m_{00}\Gamma^\ell_{m
\ell} -\Gamma^0_{0m}\Gamma^m_{00})+\mathcal R-V\Gamma^m_{0m} ,\end{align*}}where we have employed the relation $F^2\Gamma^m_{0\ell}\Gamma^\ell_{0m}+g^{ij}\Gamma^0_{im}\Gamma^m_{j0}=0 $ in the penultimate step. Exploiting the identities itemized in Proposition~\ref{prop:conn0}, along with the relation $ \Gamma^\ell_{m\ell}=\partial_m\log\sqrt g$, we may  reduce Eq.~\ref{eq:F_flat} into\begin{align*}0={}&F^{2}\frac{\partial^{2}}{\partial\psi^{2}}\log F-\frac{\partial V}{\partial\psi}+F^2\left[ \frac{\partial }{\partial u^{m}} \left( -\frac{1}{F} g^{jm}\frac{\partial }{\partial u^{j}}\frac{1}{F} \right)+\frac{1}{F^{3}}g^{jm}\Gamma^\ell_{m
\ell}\frac{\partial F}{\partial u^{j}} +\frac{1}{F^{2}}g^{\ell m}\frac{\partial\log F}{\partial u^{\ell}}\frac{\partial\log F}{\partial u^{m}}\right]-V(\Gamma^m_{0m} +2\Gamma^{0}_{00})+\mathcal R\notag\\={}&F^{2}\frac{\partial^{2}}{\partial\psi^{2}}\log F-\frac{\partial V}{\partial\psi}+F^2\left[ -\frac{1}{F}\frac{\partial }{\partial u^{m}} \left(  g^{jm}\frac{\partial }{\partial u^{j}}\frac{1}{F} \right)-\frac{1}{F}\frac{g^{jm}}{\sqrt{g}}\frac{\partial \sqrt{g} }{\partial u^{m}}\frac{\partial }{\partial u^{j}} \frac{1}{F}\right]-V(\Gamma^m_{0m} +2\Gamma^{0}_{00}) +\mathcal R\notag\\={}&-F\Delta_\Sigma\frac1F+F^2\frac{\partial^2}{\partial\psi^2}\log
F-\frac{\partial V}{\partial\psi}-V(\Gamma^m_{0m} +2\Gamma^{0}_{00}) +\mathcal R\overset{\text{Eq.  }\ref{eq:Lap_lnF_mod}}{=\!\!=\!\!=\!\!=\!\!=}\Delta\log F-\frac{\partial V}{\partial\psi}-V(\Gamma^m_{0m} +\Gamma^{0}_{00})+\mathcal R,\end{align*}which completes the proof.\qed\end{proof}\begin{remark}In Eq.~\ref{eq:dlnF_R}, the expression $ \partial V/\partial\psi$ is understood as the derivative with respect to the curvilinear coordinate $\psi$, thus is well-defined for any continuously differentiable inputs $V\in C^1(\mathfrak D;\mathbb R)$.\eor
\end{remark}
\begin{proposition}[Evolution of Ricci Curvature]\label{prop:Ric_Max}We have the following formulae valid in arbitrary dimensions:\begin{align}\frac{ \partial R_{ij}}{\partial\psi}={}&-b_{ij}\Delta_\Sigma\frac1F+\frac{\mathcal Rb_{ij}-b^{\ell}_\ell R_{ij}}{F}+b^{\ell}_\ell\left(  -\partial_{i}\partial_{j}\frac{1}{F}+\Gamma_{ij}^{k}\partial_{k}\frac1F\right)+b_{j}^{k}\left(  \partial_{i}\partial_{k}\frac{1}{F}-\Gamma_{ik}^{\ell}\partial_{\ell}\frac1F\right) +b_{i}^{k}\left(  \partial_{j}\partial_{k}\frac{1}{F}-\Gamma_{jk}^{\ell}\partial_{\ell}\frac1F\right) \label{eq:dRij}; \\\frac{\partial\mathcal R}{\partial\psi}={}&-\frac{2b_k^jR_{j}^k}{F}-2b^{\ell}_\ell\Delta_\Sigma\frac1F+2g^{ij}b_{j}^{k}\left(  \partial_{i}\partial_{k}\frac{1}{F}-\Gamma_{ik}^{\ell}\partial_{\ell}\frac1F\right)\label{eq:dR_phi}\\={}&-\frac{2G^{ij}b_{ij}}{F}-\frac{b^{\ell}_\ell\mathcal R}{F}-\frac{2}{\sqrt{g}}\partial_i\left( \beta^{ik}\sqrt g\partial_k\frac1F \right),\tag{\ref{eq:dR_phi}*}\label{eq:dR_phi1}\end{align}where $R^k_j:=g^{ik}R_{ij},\beta^{ik}:=b^\ell_\ell g^{ik}-g^{ij}b^k_j$ and $ G^{ij}=R^{ij}-\frac12g^{ij}\mathcal R$ is the (contravariant) Einstein tensor with $ R^{ij}=g^{ik}R_k^j$.
\end{proposition}\begin{proof}Using Eqs.~\ref{eq:bij_evolv} and \ref{eq:H_evolv} from Proposition~\ref{prop:bij_H},
we may differentiate both sides of the identity $R_{ij}=b^{k}_kb_{ij}-b_{ki}g^{km}b_{mj} $ to obtain{\allowdisplaybreaks[3]\begin{align*}\frac{ \partial R_{ij}}{\partial\psi}={}&\left(-\Delta_\Sigma\frac1F-\frac{b^i_jb^j_i}{F}\right)b_{ij}+b^{\ell}_\ell\left( \frac{b_{ki}g^{km}b_{mj}}{F} -\partial_{i}\partial_{j}\frac{1}{F}+\Gamma_{ij}^{k}\partial_{k}\frac1F\right)\notag\\&+\frac{2b_{ki}g^{k\ell}b^m_{\ell}b_{mj}  }{F}-\left( \frac{b_{\ell i}g^{\ell p}b_{pk}}{F} -\partial_{i}\partial_{k}\frac{1}{F}+\Gamma_{ik}^{\ell}\partial_{\ell}\frac1F\right) g^{km}b_{mj}-\left( \frac{b_{\ell j}g^{\ell p}b_{pk}}{F} -\partial_{j}\partial_{k}\frac{1}{F}+\Gamma_{jk}^{\ell}\partial_{\ell}\frac1F\right) g^{km}b_{mi}\notag\\={}&-b_{ij}\Delta_\Sigma\frac1F+\frac{\mathcal R-b^{\ell}_\ell b^k_k}{F}b_{ij}+b^{\ell}_\ell\left( \frac{b^{k}_kb_{ij}-R_{ij}}{F} -\partial_{i}\partial_{j}\frac{1}{F}+\Gamma_{ij}^{k}\partial_{k}\frac1F\right)\notag\\&+\frac{2b_{ki}g^{k\ell}(b^{m}_{m}b_{\ell j}-R_{\ell j})  }{F}-\left( \frac{b^{m}_{m}b_{ik}-R_{ik}}{F} -\partial_{i}\partial_{k}\frac{1}{F}+\Gamma_{ik}^{\ell}\partial_{\ell}\frac1F\right) b_{j}^k-\left( \frac{b^{m}_{m}b_{jk}-R_{jk}}{F} -\partial_{j}\partial_{k}\frac{1}{F}+\Gamma_{jk}^{\ell}\partial_{\ell}\frac1F\right) b^{k}_i\notag\\={}&-b_{ij}\Delta_\Sigma\frac1F+\frac{\mathcal Rb_{ij}-b^{\ell}_\ell R_{ij}}{F}+b^{\ell}_\ell\left(  -\partial_{i}\partial_{j}\frac{1}{F}+\Gamma_{ij}^{k}\partial_{k}\frac1F\right)+b_{j}^{k}\left(  \partial_{i}\partial_{k}\frac{1}{F}-\Gamma_{ik}^{\ell}\partial_{\ell}\frac1F\right) +b_{i}^{k}\left(  \partial_{j}\partial_{k}\frac{1}{F}-\Gamma_{jk}^{\ell}\partial_{\ell}\frac1F\right) ,\end{align*}}thus proving Eq.~\ref{eq:dRij}. (Here, we have used the relation $ b^{m}_{i}b_{mj}=b_{ki}g^{km}b_{mj}=b^{k}_{j}b_{ki}$, which also implies the symmetry $b^{\ell }_iR_{\ell j}=b^{\ell}_jR_{\ell i} $.)

Carrying this further, we have\begin{align*}\frac{\partial\mathcal R}{\partial\psi}=-\frac{2g^{ik}b_k^jR_{ij}}{F}+g^{ij}\frac{ \partial R_{ij}}{\partial\psi}=-\frac{2b_k^jR_{j}^k}{F}-2b^{\ell}_\ell\Delta_\Sigma\frac1F+2g^{ij}b_{j}^{k}\left(  \partial_{i}\partial_{k}\frac{1}{F}-\Gamma_{ik}^{\ell}\partial_{\ell}\frac1F\right),\end{align*}as claimed in Eq.~\ref{eq:dR_phi}. Using covariant derivatives indicated by semicolons, we may rewrite the right-hand side of  Eq.~\ref{eq:dR_phi} as\[-\frac{2b_k^jR_{j}^k}{F}-2\beta^{ik}\left( \frac1F\right)_{\raisebox{4pt}{$_{;ik}$}}=-\frac{2b_k^jR_{j}^k}{F}-2\Biggl(\beta^{ik}\partial_{i}\frac1F\Biggr)_{\raisebox{4.2pt}{$_{;k}$}}+2(\beta^{ik})_{;k}\partial_{i}\frac1F=-\frac{2b_k^jR_{j}^k}{F}-\frac{2}{\sqrt{g}}\partial_i\left( \beta^{ik}\sqrt g\partial_k\frac1F \right)+2(\beta^{ik})_{;k}\partial_{i}\frac1F,\]but\begin{align}\label{eq:beta_div_free}(\beta^{ik})_{;k}=b^\ell_{\ell;k} g^{ik}-g^{ij}b^k_{j;k}=b^\ell_{\ell;k} g^{ik}-g^{ij}b^k_{k;j}=0.\end{align}This leads to a succinct proof of Eq.~\ref{eq:dR_phi1}, based on the fact that covariant derivatives of the metric vanishes $ (g^{ik})_{;\ell}:=\partial_\ell g^{ik}+g^{im}\Gamma^k_{m\ell}+g^{km}\Gamma^i_{m\ell}=0$, and the symmetry $ (\partial_i b^k_j+\Gamma^k_{i\ell}b^\ell_j-\Gamma^\ell_{ij}b^k_\ell=:)b^k_{j;i}=b^k_{i;j}$ which follows from the  Codazzi-Mainardi equation $\partial_i b^k_j+\Gamma^k_{i\ell}b^\ell_j=\partial_j b^k_i+\Gamma^k_{j\ell}b^\ell_i $.  \qed\end{proof}

\subsection{Diffeomorphic Invariance of the Curvature Formulae\label{subsec:diff_inv}}
Locally, a level set of $\psi$ can be identified with a level set of $f\circ\psi $, where the smooth function $f:\mathbb R\longrightarrow\mathbb R$ has non-vanishing derivative $ f'\neq0$. It is thus natural to ask if the formulae for mean and scalar curvatures (Eqs.~\ref{eq:nH} and~\ref{eq:dlnF_R}) would remain intact under such function compositions. This question is answered in the positive by the following proposition.\begin{proposition}[Diffeomorphic Invariance] For every $f\in C^3(\mathbb R;\mathbb R)$ satisfying $ f'\neq0$, Eq.~\ref{eq:nH} is dif\-feo\-mor\-phi\-cal\-ly invariant in the sense  of\begin{align}\frac{\D f/\D\psi}{|\D f/\D\psi|}\Tr(\hat W)+\frac{\partial |\nabla (f\circ\psi)|}{\partial(f\circ\psi)}=\frac{ \Delta(f\circ\psi)}{ |\nabla (f\circ\psi)|}\label{eq:2H_F_f}\end{align}where $ \Tr(\hat W)$ is evaluated with respect to the normal $ \bm n=-\nabla\psi/|\nabla\psi|$, and the dif\-feo\-mor\-phic invariance of Eq.~\ref{eq:dlnF_R} is manifested in the following way:\begin{align}
\Delta\log  |\nabla (f\circ\psi)|+\mathcal R-\nabla\cdot\left[\Delta(f\circ\psi)\frac{\nabla(f\circ\psi)}{ |\nabla (f\circ\psi)|^{2}} \right]=0.\label{eq:dlnF_R_f}
\end{align} \end{proposition}
\begin{proof}By the chain rule, we have\begin{align*}\nabla (f\circ\psi)=\frac{\D f}{\D\psi}\nabla\psi,\quad\Delta(f\circ\psi)=\frac{\D f}{\D\psi}\Delta\psi+\frac{\D ^{2}f}{\D\psi^{2}}|\nabla\psi|^{2}, \end{align*}thus we may compute\begin{align*}\frac{ \Delta(f\circ\psi)}{ |\nabla (f\circ\psi)|}-\frac{\partial |\nabla (f\circ\psi)|}{\partial(f\circ\psi)}={}&\frac{\D f/\D\psi}{|\D f/\D\psi|}\frac{\Delta\psi}{|\nabla\psi|}+\frac{\D^{2} f/\D\psi^{2}}{|\D f/\D\psi|}\vert \nabla\psi\vert-\frac{1}{\D f/\D\psi}\frac{\partial}{\partial\psi}\left\vert \frac{\D f}{\D\psi}\nabla\psi \right\vert=\frac{\D f/\D\psi}{|\D f/\D\psi|}\left(\frac{\Delta\psi}{|\nabla\psi|}-\frac{\partial |\nabla\psi|}{\partial\psi}\right),\end{align*}which confirms Eq.~\ref{eq:2H_F_f} through Eq.~\ref{eq:nH}. Next, we evaluate\begin{align*}\nabla\cdot\left[ \Delta(f\circ\psi)\frac{\nabla(f\circ\psi)}{ |\nabla (f\circ\psi)|^{2}} \right]-
\Delta\log  |\nabla (f\circ\psi)|={}&\nabla\cdot\left[ \left(\frac{\D f}{\D\psi}\Delta\psi+\frac{\D ^{2}f}{\D\psi^{2}}|\nabla\psi|^{2}\right)\frac{\dfrac{\D f}{\D\psi}\nabla\psi}{ \left|\dfrac{\D f}{\D\psi}\nabla\psi\right|^{2}} \right]-
\Delta\log\left\vert \frac{\D f}{\D\psi}\nabla\psi \right|\notag\\={}&\nabla\cdot\left[ \left(\Delta\psi+\frac{\D ^{2}f/\D\psi^{2}}{\D f/\D\psi}|\nabla\psi|^{2}\right)\frac{\nabla\psi}{ |\nabla\psi|^{2}} \right]-
\Delta\log\left\vert \frac{\D f}{\D\psi}\nabla\psi \right|\notag\\={}&\nabla\cdot\left( \Delta\psi\frac{\nabla\psi}{ |\nabla\psi|^{2}} \right)-
\Delta\log  |\nabla\psi|+\nabla\cdot\left( \frac{\D ^{2}f/\D\psi^{2}}{\D f/\D\psi}\nabla\psi \right)-
\Delta\log\left\vert \frac{\D f}{\D\psi} \right|,\end{align*}and quote the decomposition of Laplacian (Eq.~\ref{eq:F_LapOp}) to deduce\begin{align*}\nabla\cdot\left( \frac{\D ^{2}f/\D\psi^{2}}{\D f/\D\psi}\nabla\psi \right)-
\Delta\log\left\vert \frac{\D f}{\D\psi} \right|={}&|\nabla\psi|^{2}\frac{\D^{2} }{\D\psi^{2}}\log\left\vert \frac{\D f}{\D\psi} \right|+(\Delta\psi)\frac{\D }{\D\psi}\log\left\vert \frac{\D f}{\D\psi} \right|-
\Delta\log\left\vert \frac{\D f}{\D\psi} \right|\notag\\={}&|\nabla\psi|^{2}\frac{\D^{2} }{\D\psi^{2}}\log\left\vert \frac{\D f}{\D\psi} \right|+(\Delta\psi)\frac{\D }{\D\psi}\log\left\vert \frac{\D f}{\D\psi} \right|-\left[|\nabla\psi|^{2}\frac{\partial ^{2}}{\partial\psi ^{2}}+(\Delta\psi)\frac{\partial }{\partial
\psi }\right]\log\left\vert \frac{\D f}{\D\psi} \right|=0,\end{align*}which completes the verification of Eq.~\ref{eq:dlnF_R_f}.\qed\end{proof}

\section{Scalar Curvature Formulae in Low Dimensions\label{sec:low_dim}}
\subsection{Examples from Classical Analysis\label{subsec:CA}}The Ricci scalar curvature identity (Eq.~\ref{eq:dlnF_R}) in low dimensions can also be independently verified with brute force. \begin{example}[Ricci Scalar Curvature in Low Dimensions]For any real-valued, three-times continuously differentiable function $ \psi\in C^3(\Omega;\mathbb R)$  defined in a two-dimensional domain $ \Omega\subset \mathbb R^2$, one has the exact identity\begin{align}\label{eq:2D_triv_id}(\partial^2_x+\partial^2_y)\log \sqrt{(\partial_x\psi)^2+(\partial_y\psi)^2}=\partial_{x}\left[ \frac{\mathcal (\partial_x^2\psi+\partial_y^2\psi)\partial_{x}\psi}{ {(\partial_x\psi)^2+(\partial_y\psi)^2}} \right]+\partial_{y}\left[\frac{\mathcal (\partial_x^2\psi+\partial_y^2\psi)\partial_{y}\psi}{ {(\partial_x\psi)^2+(\partial_y\psi)^2}} \right],\quad\text{ wherever }|\nabla\psi|\neq0,\end{align}which degenerates into the obvious relation $ \partial_x^2\log |\partial_x\psi|=\partial_x(\partial^2_x\psi/\partial_x\psi)$ when there is no dependence along the $y$-axis.

For a surface embedded   in $ \mathbb R^3$ depicted by the implicit function $ \psi(x,y,z)=0$, the Gaussian curvature $ K=\mathcal R/2$ satisfies\begin{align}
2K={}&\partial_{x}\left[ \frac{\mathcal (\partial_x^2\psi+\partial_y^2\psi+\partial_z^2\psi)\partial_{x}\psi}{ {(\partial_x\psi)^2+(\partial_y\psi)^2+(\partial_z\psi)^2}} \right]+\partial_{y}\left[\frac{\mathcal (\partial_x^2\psi+\partial_y^2\psi+\partial_z^2\psi)\partial_{y}\psi}{ {(\partial_x\psi)^2+(\partial_y\psi)^2+(\partial_z\psi)^2}} \right]+\partial_{z}\left[ \frac{\mathcal (\partial_x^2\psi+\partial_y^2\psi+\partial_z^2\psi)\partial_{z}\psi}{ {(\partial_x\psi)^2+(\partial_y\psi)^2+(\partial_z\psi)^2}} \right]\notag\\&-(\partial^2_x+\partial^2_y+\partial_z^{2})\log \sqrt{(\partial_x\psi)^2+(\partial_y\psi)^2+(\partial_z\psi)^2}.\label{eq:psi_K}
\end{align}In particular, for the Monge form $ z=f(x,y)$, one has\begin{align}
K=\frac{1}{[1+(\partial_{x}f)^{2}+(\partial_{y}f)^{2}]^{2}}\det\begin{pmatrix}\partial_{x}^2f & \partial_x\partial_yf \\
\partial_y\partial_xf & \partial_y^2f \\
\end{pmatrix}.\label{eq:K_Monge}
\end{align}\end{example}\begin{proof}Writing $ \psi(x,y)$ as a bivariate function of $ \zeta=x+iy,\overline\zeta=x-iy$, then we have $ \partial_{\raisebox{-.5pt}{$_\zeta$}}\psi=(\partial_x\psi-i\partial_y\psi)/2$ and $\partial_{\overline\zeta}\psi=(\partial_x\psi+i\partial_y\psi)/2 $, and both sides of Eq.~\ref{eq:2D_triv_id} equal $2\partial_{\raisebox{-.5pt}{$_\zeta$}}\partial_{\overline\zeta}\log (\partial_{\raisebox{-.5pt}{$_\zeta$}}\psi\partial_{\overline\zeta}\psi)$.

The right-hand side of  Eq.~\ref{eq:psi_K} is equivalent to \begin{align*}&\frac{2}{[(\partial_x\psi)^2+(\partial_y\psi)^2+(\partial_z\psi)^2]^{2}(\partial_z\psi)^2}\times\notag\\&\times\det\begin{pmatrix}\partial_{z}\psi(\partial^{2}_{x}\psi\partial_z\psi-2\partial_x\psi\partial_x\partial_z\psi)+(\partial_x\psi)^2\partial^2_z\psi & \partial_{z}\psi(-\partial_{x}\psi\partial_{y}\partial_{z}\psi+\partial_{x}\partial_{y}\psi\partial_{z}\psi-\partial_{x}\partial_{z}\psi\partial_{y}\psi)+\partial _{x}\psi\partial_{y}\psi\partial^2_z\psi \\
\partial_{z}\psi(-\partial_{x}\psi\partial_{y}\partial_{z}\psi+\partial_{x}\partial_{y}\psi\partial_{z}\psi-\partial_{x}\partial_{z}\psi\partial_{y}\psi)+\partial _{x}\psi\partial_{y}\psi\partial^2_z\psi & \partial_{z}\psi(\partial^{2}_{y}\psi\partial_z\psi-2\partial_y\psi\partial_y\partial_z\psi)+(\partial_y\psi)^2\partial^2_z\psi \\
\end{pmatrix},\end{align*}which is a standard formula in differential geometry. Taking $ \psi(x,y,z)=z-f(x,y)$, we recover Eq.~\ref{eq:K_Monge}.
\qed\end{proof}\begin{remark}Despite being a vacuous truth, the differential identity Eq.~\ref{eq:2D_triv_id} (when appropriately rewritten) can serve a good purpose for the geometric analysis of two-dimensional problems. See, for example, the applications  in \S\ref{sec:p_harmonic}.
\eor\end{remark}

\begin{example}[Planar Harmonic Function]\label{cor:planar}Let $\varphi(x,y) $ be a harmonic function in a certain domain in $ \mathbb R^2 $, then  \[\left( \frac{\partial^{2}}{\partial x^{2}} +\frac{\partial^{2}}{\partial y^{2}} \right)\log|\nabla\varphi(x,y)|=0.\]\end{example}\begin{proof}From complex analysis, one can rewrite    $|\nabla\varphi(x,y)|=|f'(\zeta)|$ where $\zeta=x+iy$
and $f(\zeta)$ is constructed from $ \varphi(x,y)$ and its harmonic conjugate. As  $\log |\nabla\varphi(x,y)|=\mathrm{Re}\log
f'(\zeta)$ is the real part of a complex analytic function, it must be a two-dimensional harmonic function. \qed\end{proof}\begin{example}[Conical Harmonic Function]\label{cor:CES}In spherical coordinates $ \bm r=(r\sin\theta\cos\phi,r\sin\theta\sin\phi,r\cos\phi)$, the function $ \varphi(\bm r)=h(\tan\frac\theta2,\phi)$ is harmonic so long as $h(\rho,\phi) $ defines a two-dimensional harmonic function in polar coordinates: $ (\partial^2_\rho+\rho^{-1}\partial_\rho+\rho^{-2}\partial^2_\phi)h(\rho,\phi)=0$. The level set of $ \varphi(\bm r)$ is a conical surface (necessarily with vanishing scalar curvature) excluding the apex $ \bm r=0$. The following identity is satisfied for $ |\bm r|\neq0$:\[\Delta\log|\nabla\varphi(\bm r)|=\frac{1}{r^2}\left[ \frac{\partial}{\partial r}\left( r^{2} \frac{\partial}{\partial r}\right)+\frac{1}{\sin\theta}\frac{\partial}{\partial\theta} \left( {\sin\theta}\frac{\partial}{\partial\theta} \right)+\frac{1}{\sin^2\theta}\frac{\partial^2}{\partial\phi^2}\right]\log |\nabla\varphi(\bm r)|=0.\]\end{example}\begin{proof}With the substitution $\rho=\tan\frac\theta2 $, we obtain\[(\D\rho)^2+\rho^2(\D\phi)^2=\frac{(\D\theta)^2+\sin^2\theta(\D\phi)^2}{4\cos^4\frac{\theta}{2}}\text{ and }\frac{\partial^2}{\partial\rho^2}+\frac{1}{\rho}\frac{\partial}{\partial\rho}+\frac{1}{\rho^2}\frac{\partial^2}{\partial\phi^2}=4\cos^4\frac{\theta}{2}\left[ \frac{1}{\sin\theta}\frac{\partial}{\partial\theta} \left( {\sin\theta}\frac{\partial}{\partial\theta} \right)+\frac{1}{\sin^2\theta}\frac{\partial^2}{\partial\phi^2}\right].\]Therefore, we have \[ |\nabla\varphi(\bm r)|^2=\frac{1}{r^2}\left[  \left( {}\frac{\partial h(\tan\frac\theta2,\phi)}{\partial\theta} \right)^{2}+\frac{1}{\sin^2\theta}  \left( {}\frac{\partial h(\tan\frac\theta2,\phi)}{\partial\phi} \right)^{2}\right]=\frac{1}{4r^{2}\cos^4\frac{\theta}{2}}\left[  \left( {}\frac{\partial h(\rho,\phi)}{\partial\rho} \right)^{2}+\frac{1}{\rho^{2}}  \left( {}\frac{\partial h(\rho,\phi)}{\partial\phi} \right)^{2}\right],\]and \begin{align*}\Delta\log|\nabla\varphi(\bm r)|={}&\frac{1}{r^2}\left[ \frac{\partial}{\partial r}\left( r^{2} \frac{\partial}{\partial r}\right)+\frac{1}{\sin\theta}\frac{\partial}{\partial\theta} \left( {\sin\theta}\frac{\partial}{\partial\theta} \right)\right]\log\frac{1}{2r\cos^2\frac{\theta}{2}}\notag\\&+\frac{1}{4r^2\cos^4\frac{\theta}{2}}\left( \frac{\partial^2}{\partial\rho^2}+\frac{1}{\rho}\frac{\partial}{\partial\rho}+\frac{1}{\rho^2}\frac{\partial^2}{\partial\phi^2} \right)\log\sqrt{  \left( {}\frac{\partial h(\rho,\phi)}{\partial\rho} \right)^{2}+\frac{1}{\rho^{2}}  \left( {}\frac{\partial h(\rho,\phi)}{\partial\phi} \right)^{2}}\notag\\={}&\frac{1}{4r^2\cos^4\frac{\theta}{2}}\left( \frac{\partial^2}{\partial\rho^2}+\frac{1}{\rho}\frac{\partial}{\partial\rho}+\frac{1}{\rho^2}\frac{\partial^2}{\partial\phi^2} \right)\log\sqrt{  \left( {}\frac{\partial h(\rho,\phi)}{\partial\rho} \right)^{2}+\frac{1}{\rho^{2}}  \left( {}\frac{\partial h(\rho,\phi)}{\partial\phi} \right)^{2}}.\end{align*}Now, quoting the result in Example~\ref{cor:planar}, we see that $ \Delta\log|\nabla\varphi(\bm r)|=0$.\qed\end{proof}
\subsection{Examples from Electrostatics\label{subsec:CP}}
Many worked examples in electrostatics provide analytic expressions for the geometry of equipotential surfaces (level sets of the electrostatic potential $ \varphi(\bm r)$, which is a harmonic function) and the corresponding spatial distributions of the electric field intensity $E(\bm r)=|\nabla\varphi(\bm r)| $.  Several specific cases listed below thus serve as direct verifications of the  relation  $ \Delta\log E(\bm r)+2K(\bm r)=0$, where $ K(\bm r)=\mathcal R(\bm r)/2$ is the Gaussian curvature of the equipotential surface. \begin{example}[Isolated Ellipsoidal Conductor]\label{cor:IEC}For an isolated conducting ellipsoid surface\begin{align}\frac{x^2}{a^2}+\frac{y^2}{b^2}+\frac{z^2}{c^2}=1
\quad(a>b>c)\label{eq:ellipsoid}\end{align}with total charge $Q$, the surface charge is distributed as \cite[Ref.][p.~22]{Landau}\[\sigma(\bm r)=\frac{Q}{4\pi
abc}\left/\sqrt{\frac{x^2}{a^4}+\frac{y^2}{b^4}+\frac{z^2}{c^4}}\right.=\frac{Q}{4\pi}\sqrt[4]{\frac{K(\bm r)}{a^2b^2c^2}},\]
 thus on the conductor surface (with $\bm r(x,y,z) $ fulfilling Eq.~\ref{eq:ellipsoid}), the electrostatic field intensity satisfies $\log E=\log\sqrt[4]{K}+\mathrm{const}$ and
 the following identity is satisfied:\[\left( \frac{x^2}{a^4}+\frac{y^2}{b^4}+\frac{z^2}{c^4} \right)^2\left( \frac{\partial^{2}}{\partial x^{2}} +\frac{\partial^{2}}{\partial y^{2}} +\frac{\partial^{2}}{\partial z^{2}}\right)\log E(\bm r)+\frac{2}{a^{2}b^{2}c^2}=0.\]\end{example}\begin{proof}Following the convention in Ref.~\cite{Landau}, we introduce the curvilinear
coordinates $ \bm r=\bm r(\xi,\eta,\zeta)$ ($\xi\geq-c^2\geq\eta\geq-b^2\geq\zeta\geq-a^2$) as
the three roots of the following cubic equation in $u$:\[\frac{x^2}{a^2+u}+\frac{y^2}{b^2+u}+\frac{z^2}{c^2+u}=1.\]From
the argument in \cite[Ref.][\S4]{Landau}, the electrostatic potential is given by\footnote{We have converted the statements in \cite[Ref.][\S4]{Landau} to SI units, wherever applicable.} \[4\pi\epsilon_{0}\varphi(\xi,\eta,\zeta)=\frac{Q}{2}\int_\xi^{+\infty}\frac{\D s}{\sqrt{(s+a^2)(s+b^2)(s+c^2)}}\]and field intensity is distributed
as \[\log E(\xi,\eta,\zeta)=-\log\sqrt{(\xi-\eta)(\xi-\zeta)}+\text{const},\]and
the equipotential surfaces (the $\xi=\text{const}$ surfaces), being ellipsoids confocal with the one shown
in Eq.~\ref{eq:ellipsoid}, has Gaussian curvature\[K(\xi,\eta,\zeta)=\frac{(a^2+\xi)(b^2+\xi)(c^2+\xi)}{(\xi-\eta)^2(\xi-\zeta)^2}.\]
With the Laplace operator \cite[Ref.][p.~19]{Landau}:\begin{align*}\Delta={}&\frac{4}{(\xi-\eta)(\zeta-\xi)(\eta-\zeta)}\left[(\eta-\zeta)\left(R_\xi\frac{\partial}{\partial\xi}\right)^2+(\zeta-\xi)\left(R_\eta\frac{\partial}{\partial\eta}\right)^2+(\xi-\eta)\left(R_\zeta\frac{\partial}{\partial\zeta}\right)^2\right],\notag\\R_u={}&\sqrt{(u+a^2)(u+b^2)(u+c^2)},\quad u=\xi,\eta,\zeta\end{align*}
one can verify that \[-2K=-2\frac{(a^2+\xi)(b^2+\xi)(c^2+\xi)}{(\xi-\eta)^2(\xi-\zeta)^2}=\Delta\log E\] indeed holds.\qed\end{proof}\begin{remark}
The ellipsoidal coordinate system employed in the proof becomes ill-defined for the spheroid case (where either $a=b >c$ or $a>b=c$), but the relation $ E\propto\sqrt[4]K$ and $ 2K+\Delta\log E=0$ remains valid, as the readers may check on their own.

In the fully degenerate case $a=b=c$, the equipotential surfaces form a family of
concentric spheres. On the surface of radius $r$, the Gaussian curvature
$K=1/r^2$, and the spatial distribution $\log E(r,\theta,\phi)=-2\log r+\text{const}$.
It is trivially true that \begin{align*}\Delta\log E=\frac{1}{r^2}\frac{\partial}{\partial r}\left(r^2 \frac{\partial\log E}{\partial r} \right)=-\frac{2}{r^2}=-2K.\end{align*}For  a ``point
charge'' in higher  dimensional space $ \mathbb R^{n+1}$, we have a power law  $ \log E(\bm r)=-n\log |\bm r|+\mathrm{const}$ instead of  the Coulomb inverse square law. Then we may perform  the following computations for a hypersphere $\{\bm x\in\mathbb R^{n+1}||\bm x|=r\} $ of radius $ r$:\[\Delta\log E=\frac{\partial^{2}\log E}{\partial r^2}+\frac{n}{r}\frac{\partial\log E}{\partial r}=-\frac{n(n-1)}{r^{2}},\] which is  consistent with the fact that a hypersphere of radius  $r$ in $ \mathbb R^{n+1}$ has scalar curvature $\mathcal R=n(n-1)/r^2$.
\eor\end{remark}

Below are two more examples concerning conductors at electrostatic equilibrium with external sources.\begin{example}[Rectangular Box Conductor Interacting with Electric Charges]\label{cor:box_ex}Consider a grounded rectangular box conductor at electrostatic equilibrium with a distribution of electric charges (with density $ \varrho(\bm r)$) that vanishes in a neighborhood of the box boundary:\begin{align*}\begin{cases}\nabla^2\varphi(\bm r)=-\frac{\varrho(\bm r)}{\epsilon_0}, & \bm r\in \Omega_{0}:=(0,a)\times(0,b)\times(0,c),\supp\varrho\Subset\Omega_0; \\
\varphi(\bm r)=0, & \bm r\in \partial\Omega_0. \\
\end{cases}\end{align*}For a boundary point $ \bm r\in\partial\Omega_0 $ that is not situated on the edges or corners of the rectangular box, we have $\Delta\log|\nabla\varphi(\bm r)|=0$.\end{example}\begin{proof}The electrostatic potential $ \varphi(\bm r)$ can be represented as $\epsilon_{0}\varphi(\bm r)=\iiint_{\Omega_0}\varrho(\bm r')G_D(\bm r,\bm r')\D^3\bm r' $ where the Dirichlet Green function $G_{D}(\bm r,\bm r') $ is given by\[G_{D}(\bm r,\bm r')=\frac{8}{abc}\sum_{\ell=1}^\infty\sum_{m=1}^\infty\sum_{n=1}^\infty\frac{\sin\frac{\ell\pi x}{a}\sin\frac{\ell\pi x'}{a}\sin\frac{m\pi y}{b}\sin\frac{m\pi y'}{b}\sin\frac{n\pi z}{c}\sin\frac{n\pi z'}{c}}{\frac{\ell^2}{a^2}+\frac{m^2}{b^2}+\frac{n^2}{c^2}}.\]Now, at a boundary point  $ \bm r\in (0,a)\times(0,b)\times\{c\}$ on the top side of the rectangular box, we have $ \partial^2_x\varphi(\bm r)=\partial^2_y\varphi(\bm r)=\partial^2_z\varphi(\bm r)=\partial_x\partial_y\varphi(\bm r)=0$ judging from the properties of $ G_D(\bm r,\bm r')$. In order to verify that $ \Delta\log (E^2)=0$, it is sufficient to check $ E^2\Delta(E^2)-|\nabla(E^2)|^2=0$. As we have \[E^2\Delta(E^2)=2E^{2}\Tr\{[\nabla\nabla\varphi(\bm r)]^{2}\}=4[\partial_z\varphi(\bm r)]^{2}\left\{ [\partial_{x}\partial_z\varphi(\bm r)]^{2}+[\partial_{y}\partial_z\varphi(\bm r)]^{2} \right\}\]according to the  Bochner-Weitzenb\"ock formula in Euclidean space, and \[|\nabla(E^2)|^2=\{\partial_x[|\partial_z\varphi(\bm r)|^2]\}^{2}+\{\partial_y[|\partial_z\varphi(\bm r)|^2]\}^{2}=4[\partial_z\varphi(\bm r)]^{2}\left\{ [\partial_{x}\partial_z\varphi(\bm r)]^{2}+[\partial_{y}\partial_z\varphi(\bm r)]^{2} \right\}\]for  $ \bm r\in (0,a)\times(0,b)\times\{c\}$, the identity  $ \Delta\log (E^2)=0$ is verified. Similarly, we can double-check $\Delta\log|\nabla\varphi(\bm r)|=0$ on the other five faces of the rectangular box.  \qed\end{proof}\begin{example}[Ellipsoid Conductor in Uniform External Field]Consider a grounded conducting ellipsoid surface with axes $a>b>c$ (defined by Eq.~\ref{eq:ellipsoid}), placed in uniform external field $ E_0\bm e_x$. The electrostatic potential \cite[Ref.][p.~23]{Landau}\[\varphi(\bm r)=-E_0\sqrt{\frac{(\xi+a^2)(\eta+b^2)(\zeta+c^2)}{(b^2-a^2)(c^2-a^2)}}[1-F(\xi)],\quad F(\xi)=\frac{\mathlarger{\int}_\xi^{+\infty}\dfrac{\D s}{(s+a^2)\sqrt{(s+a^2)(s+b^2)(s+c^2)}}}{\mathlarger{\int}_0^{+\infty}\dfrac{\D s}{(s+a^2)\sqrt{(s+a^2)(s+b^2)(s+c^2)}}}\]satisfies\[\Delta\log |\nabla\varphi(\bm r)|+2\frac{a^2b^2c^2}{\eta^2\zeta^2}=0\]on the conductor surface where $\xi=0$.\end{example}\begin{proof}Upon the substitution $ F(0)=1$, we have the following expression for $\xi=0 $:\begin{align*}\Delta\log E(\bm r)={}&\frac{2 }{\eta\zeta}\left\{-\frac{ a^2 b^2 c^2}{\eta\zeta}-2 a^2 b^2 c^2   \left[\frac{F''(0)}{F'(0)}\right]^2+2 a^2 b^2 c^2\frac{F'''(0)}{F'(0)}  + a^2  (b^{2}+c^{2})\frac{F''(0)}{F'(0)}-\frac{(a^2 -3c^{2})(b^{2}+c^{2})}{2c^2} + a^2 -3 \frac{b^2 c^2}{a^2} \right\}.\end{align*}With the relations\[\frac{F''(0)}{F'(0)}=-\frac{3 b^2 c^2+a^2 (b^2+c^2)}{2 a^2 b^2 c^2},\quad \frac{F'''(0)}{F'(0)}=\frac{15 b^4 c^4+6 a^2 b^2 c^2 (b^2+c^2)+a^4 (3 b^4+2 b^2 c^2+3 c^4)}{4 a^4 b^4 c^4},\]we obtain\[\Delta\log E(\bm r)=-2\frac{a^2b^2c^2}{\eta^2\zeta^2},\] which is the claimed identity.\qed\end{proof}
\section{A Simple Application to $p$-Harmonic Functions in Low Dimensions\label{sec:p_harmonic}}

We consider a class of nonlinear partial differential equations \begin{align}\nabla\cdot(|\nabla\psi|^{p-2}\nabla\psi)=0,\quad p>1.\label{eq:p_Harmonic}\end{align}A solution  $\psi\in W^{1,p}(\Omega) $ \cite{Evans1982} to the above nonlinear equation  minimizes energy functional $ \int_{\Omega}|\nabla\psi|^p\D^{d}\bm r, d>1$ (which is the $L^p$ analog of the Dirichlet form  $ \int_{\Omega}|\nabla\psi|^2\D^{d}\bm r$) with homogeneous Dirichlet boundary conditions on each connected component of  $ \partial \Omega$. Functions satisfying Eq.~\ref{eq:p_Harmonic} are referred to as $ p$-harmonic functions \cite{MaOuZhang2010}. By convention, the attention is restricted to the scenario with nowhere vanishing gradients $ \nabla\psi(\bm r)\neq\mathbf0,\bm r\in \Omega$, so that one may recast  Eq.~\ref{eq:p_Harmonic} into the form of ``inhomogeneous Laplace equation''
\begin{align}0=\Delta\psi+(p-2)\nabla\psi\cdot\nabla\log F=\Delta\psi+(p-2)F\frac{\partial F}{\partial\psi},\quad \text{where }F:=|\nabla\psi|.\end{align}Now that $V=(2-p)F\partial F/\partial\psi$, we immediately arrive at the following geometric identities, according to Eqs.~\ref{eq:nH} and \ref{eq:dlnF_R}\begin{align}
\Tr(\hat W)={}&(1-p)\frac{\partial F}{\partial\psi},\label{eq:TrW_p}\\\mathcal R={}&-\Delta\log F+(2-p)F^2\frac{\partial^{2}\log F}{\partial\psi^{2}}+(2-p)^{2}\left( \frac{\partial F}{\partial\psi} \right)^2=F\Delta_{\Sigma}\frac1F+(1-p)F^2\frac{\partial^{2}\log F}{\partial\psi^{2}}+(1-p)(2-p)\left(\frac{\partial F}{\partial\psi}\right)^2\notag\\={}&F\Delta_{\Sigma}\frac1F+(1-p)F\frac{\partial ^{2}F}{\partial\psi^{2}}+(1-p)^{2}\left(\frac{\partial F}{\partial\psi}\right)^2.\label{eq:R_p}
\end{align}Here, Eq.~\ref{eq:TrW_p} is consistent with Eq.~4.3 in Ref.~\cite{MaOuZhang2010}. From the geometric inequality $ (d-2)[\Tr(\hat W)]^2-(d-1)\mathcal R=(d-1)^2\mathrm{var}[\mathrm{spec}(\hat W)]\geq0$, we obtain\begin{align}\frac{(d-p)(p-1)}{d-1}\left(\frac{\partial F}{\partial\psi}\right)^2\geq F\Delta_{\Sigma}\frac1F+(1-p)F^2\frac{\partial^{2}\log F}{\partial\psi^{2}}.\end{align}

For $p\neq2$, we may rewrite Eq.~\ref{eq:R_p} as\begin{align}
\mathcal R=F\Delta_{\Sigma}\frac1F+\frac{1-p}{2-p}F^p\frac{\partial^2(F^{2-p})}{\partial\psi^2}=F\Delta_{\Sigma}\frac1F+F^p\frac{\partial}{\partial\psi}\frac{\Tr(\hat W)}{F^{p-1}},\label{eq:R_p_alt}
\end{align}and the rightmost expression in the equation above remains valid for $p=2$.

Now, we may test the validity and usefulness of the formulae  above in the following proposition by recovering the $ d=2$ case of  Theorem 1.2 in Ref.~\cite{MaOuZhang2010}.  \begin{proposition}[Extremal Properties of Planar $p$-Harmonic Functions]Let $ \psi(\bm r),\bm r\in \Omega\subset\mathbb R^d$ be a $p$-harmonic function, with $p>1$. Then its level sets satisfy\begin{align}\frac{\partial}{\partial\psi}\frac{\mathcal R}{F^p}=\frac{1}{F^{p-1}\sqrt g}\partial_{i}\left[ \frac{1}{p-1}\left(g^{ij}F^{p-1}\sqrt{g}\partial_j\frac{\Tr(\hat W)}{F^{p+1}}\right )-\frac{2}{p}\left(g^{ik}b^j_kF^{p-2}\sqrt{g}\partial_j\frac1{F^{p}}\right)\right]+\frac{\partial^{2}}{\partial\psi^{2}}\frac{\Tr(\hat W)}{F^{p-1}},\label{eq:R_Fp_diff}\end{align}where $ (b^j_k)$ is the matrix representation of the Weingarten transformation $ \hat W$.

In particular, for $d=2$, we have the equality for   $ -F^{1-p}\Tr(\hat W)=|\nabla\psi|^{1-p}\kappa$ \begin{align}\label{eq:min_prin_p_2D}\left(\frac{\Delta_\Sigma}{p-1}+F^{2}\frac{\partial^{2}}{\partial\psi^{2}}\right)\frac{\kappa}{F^{p-1}}+2\frac{\kappa}{F^{p-1}}\frac{(p-2)^2}{(p-1)}\left(\frac{\partial\log  F}{\partial s}\right)^2=0,\end{align}wherever $ \nabla( F^{1-p}\kappa)=\mathbf0$, so that a non-negative-valued $ |\nabla\psi|^{1-p}\kappa$ attains its minimum at the boundary  of a domain. (Here in Eq.~\ref{eq:min_prin_p_2D}, $\partial /\partial s$ is the tangential  derivative with respect to arc length parameter of the level set, and $ \Delta_\Sigma=\partial^2/\partial s^2$ is the Laplace-Beltrami operator on the level set.)

Furthermore, for $d=2$, we have the equality  \begin{align}F^{\alpha}
\left(\frac{\Delta_\Sigma}{p-1}+F^{2}\frac{\partial^{2}}{\partial\psi^{2}}\right)\frac{\kappa}{F^{p-1+\alpha}}=-{}&\frac{\kappa^{3}}{F^{p-1}}\frac{\alpha(\alpha+2-p)}{(p-1)^{2}}-\frac{\kappa}{F^{p-1}}\frac{2(p-2)^2+\alpha(\alpha+3p-6)}{(p-1)}\left(\frac{\partial\log  F}{\partial s}\right)^2,\label{eq:min_prin_p_2D_alpha}
\end{align}wherever $ \nabla( F^{1-p-\alpha}\kappa)=\mathbf0$. Thus, if either of the following  \begin{align}\begin{cases}\alpha\in(-\infty,p-2]\cup[0,2-p]\cup[2(2-p),+\infty), & 1<p\leq2 \\
\alpha\in(-\infty,2(2-p)]\cup[2-p,0]\cup[ p-2,+\infty),\ & p>2
\end{cases}\label{eq:cases_alpha_p}\end{align}  happens, then a non-negative-valued $ |\nabla\psi|^{1-p-\alpha}\kappa$ attains its minimum at the boundary  of a domain. For $ 3/2\leq p\leq3$, the minimum of a non-negative function $ \kappa(\bm r),\bm r\in D\cup\partial D$ is  attained at the boundary $ \partial D$.\end{proposition}\begin{proof}We divide both sides of Eq.~\ref{eq:R_p_alt} by $F^p$ and differentiate, to arrive at\begin{align*}\frac{\partial}{\partial\psi}\frac{\mathcal R}{F^p}=\frac{\partial}{\partial\psi}\left(\frac{1}{F^{p-1}}\Delta_{\Sigma}\frac1F\right)+\frac{\partial^{2}}{\partial\psi^{2}}\frac{\Tr(\hat W)}{F^{p-1}}=\frac{1}{F^{p-1}\sqrt g}\partial_{i}\frac{\partial}{\partial\psi}\left(g^{ij}\sqrt{g}\partial_j\frac1F\right)+\frac{\partial^{2}}{\partial\psi^{2}}\frac{\Tr(\hat W)}{F^{p-1}},\end{align*}where the $p$-harmonicity has been exploited in the form of $\partial_\psi(F^{p-1}\sqrt g)=0 $. Now, we simplify\begin{align*}\frac{\partial}{\partial\psi}\left(g^{ij}\sqrt{g}\partial_j\frac1F\right)=\frac{1}{p}\frac{\partial}{\partial\psi}\left(g^{ij}F^{p-1}\sqrt{g}\partial_j\frac1{F^{p}}\right)\end{align*}with the relations $ \partial_{\psi}g^{ij}=-2g^{ik}b^j_k/F$ and $ \partial_{\psi}(F^{-p}/p)=F^{-p-1}\Tr(\hat W)/(p-1)$, to obtain\begin{align*}\frac{\partial}{\partial\psi}\left(g^{ij}\sqrt{g}\partial_j\frac1F\right)=-\frac{2}{p}\left(g^{ik}b^j_kF^{p-2}\sqrt{g}\partial_j\frac1{F^{p}}\right)+\frac{1}{p-1}\left(g^{ij}F^{p-1}\sqrt{g}\partial_j\frac{\Tr(\hat W)}{F^{p+1}}\right),\end{align*}which proves Eq.~\ref{eq:R_Fp_diff}, a formula not found  in Ref.~\cite{MaOuZhang2010}.

For $d=2$, we have $ \mathcal R\equiv0$, and $ g^{ik}b^j_k=g^{ij}\Tr(\hat W)=-g^{ij}\kappa$. This brings us from  Eq.~\ref{eq:R_Fp_diff} to{\allowdisplaybreaks[3]\begin{align}0={}&\frac{1}{F^{p-1}\sqrt g}\partial_{i}\left[ \frac{1}{p-1}\left(g^{ij}F^{p-1}\sqrt{g}\partial_j\frac{\Tr(\hat W)}{F^{p+1}}\right )-\frac{2}{p}\left(g^{ij}\Tr(\hat W)F^{p-2}\sqrt{g}\partial_j\frac1{F^{p}}\right)\right]+\frac{\partial^{2}}{\partial\psi^{2}}\frac{\Tr(\hat W)}{F^{p-1}}\notag\\={}&\frac{1}{F^{p-1}\sqrt g}\partial_{i}\left[ \frac{1}{p-1}\left(g^{ij}F^{p-3}\sqrt{g}\partial_j\frac{\Tr(\hat W)}{F^{p-1}}\right )-2\frac{p-2}{p-1}\left(g^{ij}\frac{\Tr(\hat W)}{F}\sqrt{g}\partial_j\frac1{F}\right)\right]+\frac{\partial^{2}}{\partial\psi^{2}}\frac{\Tr(\hat W)}{F^{p-1}}\notag\\={}&\frac{1}{F^2}\frac{\Delta_\Sigma(F^{1-p}\Tr(\hat W))}{p-1}+\frac{p-3}{(p-1)F^{2}}\frac{\partial\log  F}{\partial s}\frac{\partial}{\partial s}\frac{\Tr(\hat W)}{F^{p-1}}+2\frac{p-2}{p-1}\frac{\Tr(\hat W)}{F}\frac{\partial}{\partial\psi}\frac{\Tr(\hat W)}{F^{p-1}}\notag\\&-2\frac{p-2}{(p-1)F^{p-1}}\frac{\partial}{\partial s}\left[\frac{\Tr(\hat W)}{F^{p-1}}F^{p-2}\right]\frac{\partial}{\partial s}\frac1{F}+\frac{\partial^{2}}{\partial\psi^{2}}\frac{\Tr(\hat W)}{F^{p-1}}\notag\\={}&\left[\frac{1}{F^2}\frac{\Delta_\Sigma(F^{1-p}\Tr(\hat W))}{p-1}+\frac{\partial^{2}}{\partial\psi^{2}}\frac{\Tr(\hat W)}{F^{p-1}}\right]+\left[ \frac{3p-7}{(p-1)F^{2}}\frac{\partial\log  F}{\partial s}\frac{\partial}{\partial s}\frac{\Tr(\hat W)}{F^{p-1}}+2\frac{p-2}{p-1}\frac{\Tr(\hat W)}{F}\frac{\partial}{\partial\psi}\frac{\Tr(\hat W)}{F^{p-1}} \right]\notag\\&+2\frac{\Tr(\hat W)}{F^{p-1}}\frac{(p-2)^2}{(p-1)F^{2}}\left(\frac{\partial\log  F}{\partial s}\right)^2,\label{eq:min_prin_p_2D_1}\end{align}where $ s$ is the arc length of the level set, and we have evaluated $ \Delta_\Sigma(1/F)$ from  Eq.~\ref{eq:R_p_alt} (with $ \mathcal R=0$) in the penultimate step.

Now, for any real number $ \alpha$, at points where  $ \nabla[ F^{1-p-\alpha}\Tr(\hat W)]=\mathbf0$, we may transform Eq.~\ref{eq:min_prin_p_2D_1} into\begin{align*}0={}&\left[\frac{1}{F^2}\frac{\Delta_\Sigma(F^{1-p-\alpha}\Tr(\hat W)F^{\alpha})}{p-1}+\frac{\partial^{2}}{\partial\psi^{2}}\frac{\Tr(\hat W)F^{\alpha}}{F^{p-1+\alpha}}\right]+\left[ \frac{3p-7}{(p-1)F^{2}}\frac{\partial\log  F}{\partial s}\frac{\partial}{\partial s}\frac{\Tr(\hat W)F^{\alpha}}{F^{p-1+\alpha}}+2\frac{p-2}{p-1}\frac{\Tr(\hat W)}{F}\frac{\partial}{\partial\psi}\frac{\Tr(\hat W)F^{\alpha}}{F^{p-1+\alpha}} \right]\notag\\&+2\frac{\Tr(\hat W)}{F^{p-1}}\frac{(p-2)^2}{(p-1)F^{2}}\left(\frac{\partial\log  F}{\partial s}\right)^2\notag\\={}&\left[\frac{F^{\alpha}\Delta_\Sigma(F^{1-p-\alpha}\Tr(\hat W)) -\frac{\alpha F^{\alpha+1}\Tr(\hat W)}{F^{p-1+\alpha}}\Delta_\Sigma\frac{1}{F}+\alpha(\alpha+1)\frac{\Tr(\hat W)}{F^{p-1}}\left(\frac{\partial\log  F}{\partial s}\right)^2}{F^{2}(p-1)}+F^{\alpha}\frac{\partial^{2}}{\partial\psi^{2}}\frac{\Tr(\hat W)}{F^{p-1+\alpha}}+\frac{\Tr(\hat W)}{F^{p-1+\alpha}}\frac{\partial^{2}(F^{\alpha})}{\partial\psi^{2}}\right]\notag\\&+\left[ \frac{\alpha(3p-7)}{(p-1)F^{2}}\frac{\Tr(\hat W)}{F^{p+1}}\left(\frac{\partial\log  F}{\partial s}\right)^2+2\frac{p-2}{p-1}\frac{[\Tr(\hat W)]^{2}}{F^{p+\alpha}}\frac{\partial(F^{\alpha})}{\partial\psi} \right]+2\frac{\Tr(\hat W)}{F^{p-1}}\frac{(p-2)^2}{(p-1)F^{2}}\left(\frac{\partial\log  F}{\partial s}\right)^2\notag\\={}&\left[\frac{\Delta_\Sigma(F^{1-p-\alpha}\Tr(\hat W)) }{F^{2-\alpha}(p-1)}+F^{\alpha}\frac{\partial^{2}}{\partial\psi^{2}}\frac{\Tr(\hat W)}{F^{p-1+\alpha}}\right]+\frac{1}{F^{2}(p-1)}\left[\frac{\alpha[\Tr(\hat W)]^{2}}{F^{p+\alpha-2}}\frac{\partial(F^{\alpha})}{\partial\psi}-\frac{\alpha[\Tr(\hat W)]^{2}}{F^{2(p-2+\alpha)}}\frac{\partial(F^{2(\alpha-1)+p})}{\partial\psi}\right]\notag\\&+\left[ 2\frac{p-2}{p-1}\frac{[\Tr(\hat W)]^{2}}{F^{p+\alpha}}\frac{\partial(F^{\alpha})}{\partial\psi} \right]+\frac{\Tr(\hat W)}{F^{p-1}}\frac{2(p-2)^2+\alpha(\alpha+1)+\alpha(3p-7)}{(p-1)F^{2}}\left(\frac{\partial\log  F}{\partial s}\right)^2\notag\\={}&\left[\frac{\Delta_\Sigma(F^{1-p-\alpha}\Tr(\hat W)) }{F^{2-\alpha}(p-1)}+F^{\alpha}\frac{\partial^{2}}{\partial\psi^{2}}\frac{\Tr(\hat W)}{F^{p-1+\alpha}}\right]-\frac{[\Tr(\hat W)]^{3}}{(p-1)^{2}}\frac{\alpha^{2}-2\alpha(\alpha-1)-\alpha p+2\alpha(p-2)}{F^{p+1}}\notag\\&+\frac{\Tr(\hat W)}{F^{p-1}}\frac{2(p-2)^2+\alpha(\alpha+1)+\alpha(3p-7)}{(p-1)F^{2}}\left(\frac{\partial\log  F}{\partial s}\right)^2,\end{align*}}as claimed in Eq.~\ref{eq:min_prin_p_2D_alpha}. Solving the inequalities $ \alpha(\alpha+2-p)\geq0$ and $ 2(p-2)^2+\alpha(\alpha+3p-6)\geq0$, we obtain  Eq.~\ref{eq:cases_alpha_p}. Setting $ \alpha=1-p$ in Eq.~\ref{eq:cases_alpha_p}, we reach the range $ 3/2\leq p\leq3$. The conclusions drawn from Eqs.~\ref{eq:min_prin_p_2D}-\ref{eq:cases_alpha_p} agree with Ref.~\cite{MaOuZhang2010}. \qed\end{proof}

\section{Application to Einstein-Hilbert Actions of High-Dimensional Equipotential Surfaces\label{sec:harmonic_sect_curv}}

During the study of a physical problem that will be described elsewhere, we encountered the following quantity\begin{align*}\oint_{\Sigma_\varphi}\mathcal R\D m(\Sigma_{\varphi})\end{align*}which is the Einstein-Hilbert action of a curved equipotential surface $ \Sigma_\varphi$ (\textit{i.e.~}level set of a harmonic function with level value $ \varphi$). For $ \dim(\Sigma_\varphi)=1$, the Einstein-Hilbert action vanishes identically.  For $ \dim(\Sigma_\varphi)=2$, the Einstein-Hilbert action is equal to $ 4\pi$ times the Euler-Poincar\'e characteristic $ \chi(\Sigma_\varphi)$, so it remains a local constant of $ \varphi$  so long as the field lines are not broken.

In the following proposition, we will present a convexity result for the less trivial cases where  $ \dim(\Sigma_\varphi)\geq3$. To emphasize that we are handling harmonic functions, we will replace the earlier notation $ F=|\nabla\psi|$ with $ E=|\nabla\varphi|$. Here, $ \varphi(\bm r)$ can be regarded as a high-dimensional ``electrostatic potential'', and $ E(\bm r)=|\nabla\varphi(\bm r)|$ the corresponding ``electric field intensity''.
Still, it will be tacitly assumed that $   E=|\nabla\varphi|$ is non-vanishing.

\begin{proposition}[Conditional Convexity of Einstein-Hilbert Action] The Einstein-Hilbert action \[\mathcal A(\varphi):=\oint_{\Sigma_\varphi}\mathcal R\D m(\Sigma_{\varphi})=\idotsint_{\Sigma_\varphi}\mathcal R\sqrt g\D u^1\cdots\D u^n\]satisfies\begin{align}\label{eq:A''}\frac{\mathcal A''(\varphi)}{6}=\oint_{\Sigma_\varphi}\frac{\frac{1}{3}b^{m}_mG^{ij}b_{ij}-R^{ij}R_{ij}+\frac12\mathcal R^2}{E^{2}}\D m(\Sigma_{\varphi})-\oint_{\Sigma_\varphi}G_{ik}\partial_i\frac{1}{E}\partial_k\frac1E\D m(\Sigma_{\varphi}).\end{align} Accordingly, for $ n\geq3$, we have $ \mathcal A''(\varphi)\geq0$ if  $\Sigma_\varphi$ has non-negative sectional curvature.
\end{proposition}\begin{proof}It follows directly from Eq.~\ref{eq:dR_phi1} that\begin{align*}\mathcal A'(\varphi)=-2\oint_{\Sigma_\varphi}\frac{G^{ij}b_{ij}}{E}\D m(\Sigma_{\varphi})=2\oint_{\Sigma_\varphi}\frac{G_{ij}(\beta^{ij}-b^\ell_\ell g^{ij})}{E}\D m(\Sigma_{\varphi})=2\oint_{\Sigma_\varphi}\frac{\beta^{ij}G_{ij}}{E}\D m(\Sigma_{\varphi})-2\left( 1-\frac{n}{2} \right)\oint_{\Sigma_\varphi}\frac{b^\ell_\ell \mathcal R}{E}\D m(\Sigma_{\varphi}),\end{align*}where $G_{ij}=R_{ij}-\frac12g_{ij}\mathcal R$ is the covariant Einstein tensor. Before computing the second-order derivative $ \mathcal A''(\varphi)$, we need{\allowdisplaybreaks[3]\begin{align}&\frac{1}{\sqrt{g}}\frac{\partial(\beta^{ik}\sqrt g)}{\partial\varphi}=\frac{b^m_m\beta^{ik}}{E}+\frac{\partial}{\partial\varphi}(b^m_mg^{ik}-g^{ij}g^{k\ell}b_{j\ell})\notag\\={}&\frac{b^m_m\beta^{ik}}{E}-g^{ik}\Biggl( \Delta_\Sigma\frac1E+\frac{b^\ell_mb^m_\ell}{E} \Biggr)-\frac{2b^m_mg^{ij}b^k_j}{E}+\frac{2g^{im}b^j_mb^k_j}{E}+\frac{2g^{km}b^\ell_mb^i_\ell}{E}-g^{ij}g^{k\ell}\left( \frac{b^m_jb_{m\ell}}{E}-\left( \frac1E\right)_{\raisebox{4pt}{$_{;j\ell}$}} \right)\notag\\={}&\frac{b^m_m(b^\ell_\ell g^{ik}-g^{ij}b^k_j)}{E}-g^{ik}\Biggl( \Delta_\Sigma\frac1E+\frac{b^\ell_mb^m_\ell}{E} \Biggr)+\Biggl(-\frac{2b^m_mg^{ij}b^k_j}{E}+\frac{2g^{im}b^j_mb^k_j}{E}\Biggr)+\left(\frac{2g^{km}b^\ell_mb^i_\ell}{E}-\frac{g^{ij}b^m_jb^{k}_{m}}{E}\right)+\left( \frac1E\right)^{\raisebox{-4pt}{$^{;ik}$}}\notag\\={}&-\frac{b^m_mg^{ij}b^k_j}{E}-g^{ik}\Biggl( \Delta_\Sigma\frac1E-\frac{\mathcal  R}{E} \Biggr)-\frac{2g^{im}R^k_m}{E}+\frac{g^{km}b^\ell_mb^i_\ell}{E}+\left( \frac1E\right)^{\raisebox{-4pt}{$^{;ik}$}}\notag\\={}&-g^{ik}\Biggl( \Delta_\Sigma\frac1E-\frac{\mathcal  R}{E} \Biggr)-\frac{3R^{ik}}{E}+\left( \frac1E\right)^{\raisebox{-4pt}{$^{;ik}$}}=-g^{ik}\Biggl( \Delta_\Sigma\frac1E+\frac{\mathcal  R}{2E} \Biggr)-\frac{3G^{ik}}{E}+\left( \frac1E\right)^{\raisebox{-4pt}{$^{;ik}$}}.\label{eq:d_beta_ij}\end{align}}Thus, we may evaluate\begin{align*}&\frac{\D}{\D\varphi}\oint_{\Sigma_\varphi}\frac{\beta^{ij}G_{ij}}{E}\D m(\Sigma_{\varphi})=\oint_{\Sigma_\varphi}\frac{G_{ij}}{E}\frac{1}{\sqrt{g}}\frac{\partial(\beta^{ij}\sqrt g)}{\partial\varphi}\D m(\Sigma_{\varphi})+\oint_{\Sigma_\varphi}\beta^{ij}\frac{\partial}{\partial\varphi}\frac{G_{ij}}{E}\D m(\Sigma_{\varphi})\notag\\={}&\left(\frac n2-1\right)\oint_{\Sigma_\varphi}\frac{\mathcal R}{E}\Delta_\Sigma\frac1E\D m(\Sigma_{\varphi})-3\oint_{\Sigma_\varphi}\frac{R^{ij}R_{ij}}{E^{2}}\D m(\Sigma_{\varphi})+\frac{5-n}{2}\oint_{\Sigma_\varphi}\frac{\mathcal R^2}{E^{2}}\D m(\Sigma_{\varphi})\notag\\&+\oint_{\Sigma_\varphi}\frac{G_{ik}}{E}\left( \frac1E\right)^{\raisebox{-4pt}{$^{;ik}$}}\D m(\Sigma_{\varphi})+\oint_{\Sigma_\varphi}\beta^{ij}\frac{\partial}{\partial\varphi}\frac{G_{ij}}{E}\D m(\Sigma_{\varphi}).\end{align*}where we have used the squared modulus of the Einstein tensor $ G^{ij}G_{ij}=R^{ij}R_{ij}+(\frac n4-1)\mathcal R^{2}$  in the last step.
From Eqs.~\ref{eq:dRij} and \ref{eq:dR_phi1}, it is not hard to see that \begin{align*}\beta^{ij}\frac{\partial}{\partial\varphi}\frac{R_{ij}}{E}={}&\frac{2G_{ik}}{E}\left( \frac1E\right)^{\raisebox{-4pt}{$^{;ik}$}}+\frac{\mathcal R^{2}}{E^{2}}-\frac{b^{\ell}_\ell}{E\sqrt{g}}\partial_i\left( \beta^{ik}\sqrt g\partial_k\frac1E \right),\notag\\\beta^{ij}\frac{\partial}{\partial\varphi}\frac{\frac12g_{ij}\mathcal R}{E}={}&\frac{\mathcal R^{2}}{E^{2}}+\frac{(n-1)b^{m}_m}{2}\frac{\partial}{\partial\varphi}\frac{\mathcal R}{E}=-\frac{\mathcal R^{2}}{E^{2}}-\frac{(n-1)b^{m}_mG^{ij}b_{ij}}{E^{2}}-\frac{(n-1)b^{m}_m}{E\sqrt{g}}\partial_i\left( \beta^{ik}\sqrt g\partial_k\frac1E \right),\end{align*}so we have\begin{align*}&\frac{\D}{\D\varphi}\oint_{\Sigma_\varphi}\frac{\beta^{ij}G_{ij}}{E}\D m(\Sigma_{\varphi})=\oint_{\Sigma_\varphi}\frac{G_{ij}}{E}\frac{1}{\sqrt{g}}\frac{\partial(\beta^{ij}\sqrt g)}{\partial\varphi}\D m(\Sigma_{\varphi})+\oint_{\Sigma_\varphi}\beta^{ij}\frac{\partial}{\partial\varphi}\frac{G_{ij}}{E}\D m(\Sigma_{\varphi})\notag\\={}&\left(\frac n2-1\right)\oint_{\Sigma_\varphi}\frac{\mathcal R}{E}\Delta_\Sigma\frac1E\D m(\Sigma_{\varphi})+(n-1)\oint_{\Sigma_\varphi}\frac{b^{m}_mG^{ij}b_{ij}}{E^{2}}\D m(\Sigma_{\varphi})-3\oint_{\Sigma_\varphi}\frac{R^{ij}R_{ij}}{E^{2}}\D m(\Sigma_{\varphi})+\frac{5-n}{2}\oint_{\Sigma_\varphi}\frac{\mathcal R^2}{E^{2}}\D m(\Sigma_{\varphi})\notag\\&+\oint_{\Sigma_\varphi}\frac{3G_{ik}}{E}\left( \frac1E\right)^{\raisebox{-4pt}{$^{;ik}$}}\D m(\Sigma_{\varphi})+(n-2)\oint_{\Sigma_\varphi}\frac{b^{\ell}_\ell}{E\sqrt{g}}\partial_i\left( \beta^{ik}\sqrt g\partial_k\frac1E \right)\D m(\Sigma_{\varphi}) \end{align*}as a consequence. Meanwhile,\begin{align*}\frac{\D}{\D\varphi}\oint_{\Sigma_\varphi}\frac{b^\ell_\ell \mathcal R}{E}\D m(\Sigma_{\varphi})=\oint_{\Sigma_\varphi}\frac{b^\ell_\ell }{E}\left[ -\frac{2G^{ij}b_{ij}}{E}-\frac{2}{\sqrt{g}}\partial_i\left( \beta^{ik}\sqrt g\partial_k\frac1E \right) \right]\D m(\Sigma_{\varphi})+\oint_{\Sigma_\varphi}\mathcal R\left( -\Delta_{\Sigma}\frac1E+\frac{\mathcal R }{E}\right)\D m(\Sigma_{\varphi})\end{align*}can be combined into the previous results to yield\begin{align*}\frac{\mathcal A''(\varphi)}{2}={}&\frac{\D}{\D\varphi}\oint_{\Sigma_\varphi}\frac{\beta^{ij}G_{ij}}{E}\D m(\Sigma_{\varphi})-\left( 1-\frac{n}{2} \right)\frac{\D}{\D\varphi}\oint_{\Sigma_\varphi}\frac{b^\ell_\ell \mathcal R}{E}\D m(\Sigma_{\varphi})\notag\\={}&\oint_{\Sigma_\varphi}\frac{b^{m}_mG^{ij}b_{ij}}{E^{2}}\D m(\Sigma_{\varphi})-3\oint_{\Sigma_\varphi}\frac{R^{ij}R_{ij}}{E^{2}}\D m(\Sigma_{\varphi})+\frac{3}{2}\oint_{\Sigma_\varphi}\frac{\mathcal R^2}{E^{2}}\D m(\Sigma_{\varphi})+\oint_{\Sigma_\varphi}\frac{3G_{ik}}{E}\left( \frac1E\right)^{\raisebox{-4pt}{$^{;ik}$}}\D m(\Sigma_{\varphi}).\end{align*}After we apply the contracted Bianchi identity $ (G_{ik})^{;i}=0$, Eq.~\ref{eq:A''}  emerges.

 In the principal curvature coordinates, the Einstein tensor behaves as\begin{align*}(G_{ij})_{1\leq i,j\leq n}=\begin{pmatrix}-g_{11}\mathlarger{\sum_{\substack{\ell<m\\\ell\neq1,m\neq1}}}k_\ell k_m &  & \multicolumn{2}{c}{\raisebox{-2.1ex}[0pt]{\Huge0}} \\
 & \ddots\ &  \\\multicolumn{2}{c}{\raisebox{-3.3ex}[0pt]{\Huge0}}
 &  & -g_{nn}\mathlarger{\sum_{\substack{\ell<m\\\ell\neq n,m\neq n}}}k_\ell k_m \\
\end{pmatrix},\end{align*}which is clearly a negative semidefinite matrix if $ \Sigma_\varphi$ has non-negative sectional curvature. In this case, we then have \begin{align}\frac{\mathcal A''(\varphi)}{6}\geq\oint_{\Sigma_\varphi}\frac{\frac{1}{3}b^{m}_mG^{ij}b_{ij}-R^{ij}R_{ij}+\frac12\mathcal R^2}{E^{2}}\D m(\Sigma_{\varphi}).\label{eq:A''_ineq}\end{align}

However, we may use principal curvatures to rewrite \begin{align}\frac{1}{3}b^{m}_mG^{ij}b_{ij}-R^{ij}R_{ij}+\frac12\mathcal R^2=\sum_{\substack{i<j\\i\neq\ell,j\neq\ell}}k_ik_jk_\ell^2+8\sum_{\substack{i<j<\ell<m}}k_ik_jk_\ell k_{m},\label{eq:comb_id}\end{align}where the second summand is left out for $n=3$. Such an identity as Eq.~\ref{eq:comb_id} is elementary though probably not self-evident. By inspection, the left-hand side of  Eq.~\ref{eq:comb_id} is a homogeneous quartic polynomial in the variables $ (k_1,\dots,k_n)$, where each variable cannot appear in a power higher than two. Terms like $ k_\ell^2k_m^2$ (where $\ell<m $) will not count towards the net sum, because $ -R^{ij}R_{ij}$ contributes $ -2k_\ell^2k_m^{2}$ while $ \frac12\mathcal R^2$ yields $ 2k^2_\ell k^2_m$ in their respective developments. Therefore, by symmetry,  the left-hand side of  Eq.~\ref{eq:comb_id} can be put into the form\begin{align*}a\sum_{\substack{i<j\\i\neq\ell,j\neq\ell}}k_ik_jk_\ell^2+b\sum_{\substack{i<j<\ell<m}}k_ik_jk_\ell k_{m},\end{align*}for some constants $ a$ and $b$. To determine that $ a=1$, simply  set $ k_1=k_2=k_3=1$ and $ k_4=\cdots=k_n=0$. To find out that $b=8 $, one may consider the special case where $ k_1=\cdots=k_n=1$.

To summarize,  Eqs.~\ref{eq:A''_ineq}  and \ref{eq:comb_id} explicitly demonstrate the non-negativity of $ \mathcal A''(\varphi)$ for equipotential surfaces $\Sigma_\varphi$ with non-negative sectional curvature.\qed\end{proof}\begin{remark}In the proof of the proposition above, we have used the following alternative forms of Eqs.~\ref{eq:nH} and \ref{eq:H_evolv}:\begin{align*}\Tr(\hat W)+\frac{\partial E}{\partial\varphi}=0,\quad \frac{\partial
\Tr(\hat W)}{\partial\varphi}=-\Delta_\Sigma\frac1E-\frac{
[\Tr(\hat W)]^{2}-\mathcal R}{E} .\end{align*}Here, the first formula is sensitive to the harmonic condition $ \Delta\varphi(\bm r)=0$, so the conclusion of this proposition does not apply to level sets of arbitrary functions. The second formula displayed above is morally equivalent to the scalar curvature identity $ \mathcal R+\Delta\log E=0$, as can be seen from the one-line proof:\begin{align*} \mathcal R+\Delta\log E=\mathcal R-E\Delta_\Sigma\frac1E+E^2\frac{\partial ^{2}\log E}{\partial\varphi ^{2}}=\mathcal R-E\Delta_\Sigma\frac1E-E^{2}\frac{\partial}{\partial\varphi}\frac{\Tr(\hat W)}{E}=\mathcal R-E\Delta_\Sigma\frac1E-E\frac{\partial
\Tr(\hat W)}{\partial\varphi}-E^2\Tr(\hat W)\frac{\partial}{\partial\varphi}\frac{1}{E},\end{align*}which hinges on a degenerate version of  Eq.~\ref{eq:Lap_lnF_mod}.\eor\end{remark}\bibliography{ScalarCurv}
\bibliographystyle{unsrt}

\end{document}